\documentclass[11pt,reqno]{amsart}
\usepackage{amssymb,mathrsfs,graphicx,enumerate}
\usepackage{amsmath,amsfonts,amssymb,amscd,amsthm,bbm}

\usepackage[retainorgcmds]{IEEEtrantools}
\usepackage{xcolor}
\usepackage{colortbl}
\usepackage{cite}
\usepackage{subcaption}
\usepackage{nccmath}
\usepackage{bm}
\usepackage{upgreek}
\usepackage{comment}
\usepackage{epsfig}

\usepackage{kotex}

\title[The relativistic manifold Cucker-Smale model with bonding force]{Interplay of geometric constraint and  bonding force in the emergent behaviors of relativistic Cucker-Smale flocks}

\author[Ahn]{Hyunjin Ahn}
\address[Hyunjin Ahn]{\newline Department of Mathematical Sciences\newline Seoul National University, Seoul 08826, Republic of Korea}
\email{yagamelaito@snu.ac.kr}

\author[Byeon]{Junhyeok Byeon}
\address[Junhyeok Byeon]{\newline Department of Mathematical Sciences\newline Seoul National University, Seoul 08826, Republic of Korea}
\email{giugi2486@snu.ac.kr}

\author[Ha]{Seung-Yeal Ha}
\address[SeungYeal Ha]{\newline Department of Mathematical Sciences and Research Institute of Mathematics \newline Seoul National University, Seoul 08826, Republic of Korea}
\email{syha@snu.ac.kr}

\author[Yoon]{Jaeyoung Yoon}
\address[Jaeyoung Yoon]{\newline Department of Mathematical Sciences\newline Seoul National University, Seoul 08826, Republic of Korea}
\email{jyoung924@snu.ac.kr}

\newtheorem{theorem}{Theorem}[section]
\newtheorem{lemma}{Lemma}[section]
\newtheorem{corollary}{Corollary}[section]
\newtheorem{proposition}{Proposition}[section]
\newtheorem{remark}{Remark}[section]

\newtheorem{definition}{Definition}[section]

\newcommand{\bbr}{\mathbb R}

\newcommand{\vast}{\bBigg@{4}}
\newcommand{\Vast}{\bBigg@{5}}

\newcommand{\bx}{\mbox{\boldmath $x$}}
\newcommand{\by}{\mbox{\boldmath $y$}}
\newcommand{\bu}{\mbox{\boldmath $u$}}
\newcommand{\bv}{\mbox{\boldmath $v$}}

\newcommand{\bw}{\mbox{\boldmath $w$}}
\newcommand{\br}{\mbox{\boldmath $r$}}
\newcommand{\bp}{\mbox{\boldmath $p$}}
\newcommand{\bq}{\mbox{\boldmath $q$}}

\begin{document}

\date{\today}


\thanks{\textbf{Acknowledgment.} The work of S.-Y. Ha was supported by National Research Foundation of Korea(NRF-2020R1A2C3A01003881)}
\begin{abstract}
We present the relativistic analogue of the Cucker-Smale model with a bonding force on Riemannian manifold, and study its emergent dynamics. The Cucker-Smale model serves a prototype example of mechanical flocking models, and it has been extensively studied from various points of view. Recently, the authors studied collision avoidance and asymptotic flocking of the Cucker-Smale model with a bonding force on the Euclidean space. In this paper, we provide an analytical framework for collision avoidance and asymptotic flocking of the proposed model on Riemannian manifolds. Our analytical framework is explicitly formulated in terms of system parameters,  initial data and the injectivity radius of the ambient manifold, and we study how the geometric information of an ambient manifold can affect the flocking dynamics.
\end{abstract}

\maketitle \centerline{\date}


\section{Introduction} \label{sec:1}
\setcounter{equation}{0}
The purpose of this paper is to continue the studies begun in \cite{A-B-H-Y, P-K-H} on the design of spatial patterns using the Cucker-Smale flocking model. Emergent behaviors of many-body systems often appear in nature, e.g., aggregation of bacteria \cite{T-B}, flocking of birds \cite{C-S}, swarming of fish \cite{D-M1, T-T}, synchronization of fireflies and pacemaker cells \cite{B-B, Er, Wi2}, etc. We refer to survey papers and a book \cite{A-B, A-B-F, C-H-L, F-T-H, O2, P-R-K, Str, VZ, Wi1} for an introduction. In this paper, we are mainly concerned with the flocking behaviors in which particles move with the common velocity by using limited environmental information and simple rules. After the seminal work by Vicsek et al. \cite{V-C-B-C-S} on the mathematical modeling of flocking, several mathematical models have been addressed in previous literature.  Among them, our main interest lies in the Cucker-Smale model \cite{C-S} which is a Newton-like model for mechanical observables such as position and velocity, and it has been extensively studied from the various points of view in the last decade, to name a few, the mean-field limit \cite{A-H-K, A-H-K-S-S, HKMRZ, H-L, H-K-Z}, the kinetic description \cite{C-F-R-T, H-T}, hydrodynamic description \cite{F-K, H-K-K, K-M-T1} and asymptotic collective behaviors in $\mathbb{R}^d$\cite{C-F-R-T,C-D-P,C-H-H-J-K2,C-H,C-L,C-H-L,C-K-P-P}, etc. In this work, we are interested in the collective behaviors of relativistic Cucker-Smale(RCS) ensembles resulting from the interplay between {\it bonding force field} and geometric constraints. To set the stage, we first begin with the classical CS model  in \cite{A-B-H-Y} with a bonding force.  

Let $\bx_i$ and $\bv_i$ be the position and velocity of 
the $i$-th particle, respectively, and the nonnegative parameter $R_{ij}^{\infty}$ denotes the preassigned asymptotic relative distance between the $i$ and $j$-th particles satisfying the following relations:
\begin{equation} \label{PA-0}
R_{ii}^{\infty} = 0, \quad i \in [N] := \{1, \cdots, N \}, \qquad R_{ij}^{\infty} =  R_{ji}^{\infty}, \quad 1 \leq i \neq j \leq N.
\end{equation}	
In this setting, the Cauchy problem to the generalized CS model with a bonding force reads as follows.
	\begin{equation}
	\begin{cases} \label{PA-1}
	\displaystyle {\dot{\bx}_i} =\bv_i,\quad t>0,\quad  i \in [N],\\
	\displaystyle {\dot{\bv}_i} =\frac{\kappa_0}{N}\sum_{j=1}^{N}\phi( |\bx_i-\bx_j|)\left(\bv_j- \bv_i\right)  \\
	\displaystyle \hspace{0.5cm} + \frac{1}{2N}\sum_{\substack{j=1 \\ j \neq  i }}^{N} 
	\Big[ \kappa_1 \Big \langle \bv_j-\bv_i, \frac{\bx_j-\bx_i}{ |\bx_j - \bx_i|}  \Big \rangle + \kappa_2\Big( |\bx_i-\bx_j|-R^{\infty}_{ij} \Big) \Big] \frac{(\bx_j-\bx_i)}{| \bx_j - \bx_i|}, \\
	\displaystyle (\bx_i, \bv_i)(0) =  (\bx^0_i, \bv^0_i) \in \bbr^{2d}, 
	\end{cases}
	\end{equation}
	where $\kappa_0, \kappa_1$ and $\kappa_2$ are nonnegative constants representing the intensities of velocity alignment and bonding interactions, respectively. Here $\langle \cdot, \cdot \rangle$ and $|\cdot|$ denote the standard inner product and its associated $\ell^2$-norm in the Euclidean space $\bbr^d$ and  the kernel $\phi:\bbr_+ \to \bbr_+$ is a communication weight representing the degree of interactions between particles which satisfies the following relations:
\begin{align}
\begin{aligned} \label{PA-2}
& \phi \in C_{\mathrm{loc}}^{0,1}(\bbr_+ \cup \{0 \}), \quad 0\leq\phi(r)\leq \phi_M<\infty, \quad r\geq 0.
\end{aligned}
\end{align}
Note that in the absence of the coupling strengths in the bracket of the R.H.S. of \eqref{PA-1},  system \eqref{PA-1} reduces to the CS model. In \cite{A-B-H-Y}, authors derived  a sufficient framework for the collision avoidance and flocking dynamics for \eqref{PA-1}. We refer to \cite{A-B-H-Y, P-K-H}  and Figure \ref{CSfig} for a detailed discussion on motivation and modeling spirit. Figure \ref{CSfig} illustrates a pattern formation from the model \eqref{PA-1}, where $\{R_{ij}\}$ is chosen to design a star-shaped pattern and heart-shaped pattern in $t \in [0,5)$ and $t \in [5,10]$, respectively. Throughout the paper, we will take \eqref{PA-0} and \eqref{PA-2} as standing assumptions. In this paper, we address the following simple questions: 
\vspace{0.1cm}
\begin{itemize}
\item
What will be the relativistic counterpart of the model \eqref{PA-1} on Riemannian manifolds ? 
\vspace{0.1cm}
\item
If such a model exists, under what conditions on initial data and system parameters, can the proposed model exhibit emergent dynamics?
\end{itemize}
\vspace{0.1cm}
In \cite{A-B-H-Y}, we used $2R_{ij}^{\infty}$ instead of $R_{ij}^{\infty}$ (see the momentum equation $\eqref{PA-1}_2$).  \newline

Next, we briefly discuss our main results regarding above two questions. Before we jump to the full generality, we start with the relativistic counterpart of \eqref{PA-1} in the Euclidean space $\bbr^d$. Let $c$ be the speed of light, and $\bx_i$ and $\bv_i$ denote the position and (relativistic) velocity of the $i$-the CS particle. For a given velocity $\bv$, we introduce the associated Lorentz factor $\Gamma$ and momentum-like observable $\bw$ under suitable normalizations and ansatz for rest mass, specific heat  and internal energy:
\begin{equation} \label{A-0}
\Gamma(\bv) :=  \frac{1}{\sqrt{1-\frac{|\bv|^2}{c^2}}}, \quad F(\bv) := \Gamma \left(1+\frac{\Gamma}{c^2}\right), \quad \bw := F(\bv) \bv,
\end{equation}
and define the map $\hat{w}: B_c({\bf 0}) \to \bbr^d$ as 
\begin{equation*} \label{A-0-0}
{\hat w}(\bv) := F(\bv) \bv = \Big( \frac{c}{\sqrt{c^2 - |\bv|^2}}  + \frac{1}{c^2 - |\bv|^2} \Big) \bv. 
\end{equation*}
Then, one can check  that ${\hat w}$ is bijective and becomes identity map in the formal nonrelativistic limit $c \to \infty$,  and there exists an inverse function ${\hat v} := \hat{w}^{-1}:\bbr^d\to B_c({\bf 0})$ satisfying the following relations:
\begin{equation} \label{A-1}
 \bv =  {\hat v}(\bw).
 \end{equation}
 Now, we are ready to provide the relativistic counterpart of the Cauchy problem \eqref{PA-1} in terms of $(\bx_i, \bw_i)$ as follows. 
\begin{equation} \label{A-7}
\begin{cases}
\displaystyle {\dot \bx}_i = \hat{v}(\bw_i),\quad t>0,\quad  i \in [N],\\
\vspace{0.1cm}
\displaystyle {\dot \bw}_i  =\frac{\kappa_0}{N}\sum_{j=1}^{N}\phi(|\bx_i-\bx_j|)\left(\hat{v}(\bw_j) - \hat{v}(\bw_i) \right) \\
\vspace{0.1cm}
\displaystyle  +~\frac{1}{2N} \sum_{\substack{j=1 \\ j \neq  i }}^{N} 
	\Big[ \kappa_1 \Big \langle  \hat{v}(\bw_j)- \hat{v}(\bw_i), \frac{\bx_j-\bx_i}{| \bx_j - \bx_i|} \Big \rangle   + \kappa_2 \Big(|\bx_i-\bx_j|-R^{\infty}_{ij} \Big) \Big]  \frac{(\bx_j-\bx_i)}{| \bx_j - \bx_i|}, \\
\displaystyle (\bx_i, \bw_i)(0) =  (\bx^0_i, \bw^0_i) \in \bbr^{2d}.
\end{cases}
\end{equation}
Note that in the formal nonrelativistic limit $(c \to \infty)$, system \eqref{A-7} reduces to the classical one \eqref{PA-1}. In fact, we will justify this formal limit rigorously in Section \ref{sec:4}. We also note that the R.H.S. of $\eqref{A-7}_2$ contains a term $| \bx_j - \bx_i |$ in the denominators so that the R.H.S. can be singular at the instant when two particles collide $(\bx_i = \bx_j)$. So in order to guarantee a classical smooth solution, we have to make sure the collision avoidance in finite time. For this, we provide a simple condition on the coupling strength $\kappa_2$ and initial data leading to collision avoidance (see Theorem \ref{T3.1}) and we also present a sufficient framework for the asymptotic flocking (see Theorem \ref{T3.2}):
\[ \sup_{0 \leq t < \infty} \max_{1 \leq i,j \leq N} | \bx_i(t) - \bx_j(t) | < \infty, \quad \lim_{t \to \infty} \max_{1 \leq i,j \leq N}  |\bv_j(t)-\bv_i(t) |=0. \]
When we turn off the bonding force $(\kappa_1, \kappa_2) = (0, 0)$, system \eqref{A-7} reduces to the RCS model \cite{H-K-R} which was systematically derived from the relativistic fluid model for gas mixture so that the consistency of \eqref{A-7} with the special relativity is automatically inherited from the relativistic fluid model (see \cite{H-K-R}) and in the classical limit $(c \to \infty)$, the RCS model \eqref{A-7} also reduces to the classical CS model in \cite{C-S} in any finite time interval. In fact, the Cauchy problem \eqref{A-7} with $(\kappa_1, \kappa_2) = (0, 0)$ has been studied in the third author and his collaborators,  to name a few, emergent dynamics \cite{H-K-R}, kinetic and hydrodynamic descriptions \cite{H-K-R0},  the mean-field limit \cite{A-H-K}. 

Second, we discuss the manifold extension of \eqref{A-7}.  Let $({\mathcal M}, g)$ be a connected, smooth Riemannian manifold with a metric tensor $g$ without boundary, and let $P_{ij}$ be the parallel transport from the tangent space at $\bx_j$ to the tangent space at $\bx_i$, and $\nabla$ be the Levi-Civita connection compatible with the metric tensor $g$. We refer to Section \ref{sec:5.1} for the minimum materials of differential geometry. Then, the manifold extension of \eqref{A-7} reads as follows:
\begin{equation}
\begin{cases} \label{A-8}
\displaystyle {\dot{\bx}_i} = \hat{v}(\bw_i),\quad t>0,\quad  i \in [N],\\
\displaystyle \nabla_{{\dot \bx}_i} \bw_i=\frac{\kappa_0}{N}\sum_{j=1}^{N}\phi(\bx_i, \bx_j)\left(P_{ij} \hat{v}(\bw_j) - \hat{v}(\bw_i) \right)\\
\displaystyle + \frac{1}{2N}  \sum_{\substack{j=1 \\ j \neq  i }}^{N} \Big[ \kappa_1 g_{{\scriptsize \bx_i}} \Big( P_{ij} \hat{v}(\bw_j) -\hat{v}(\bw_i), \frac{\log_{\bx_i}\bx_j)}{d(\bx_j, \bx_i)} \Big) + \kappa_2 \Big (d(\bx_j, \bx_i) -R^\infty_{ij} \Big)  \Big] \frac{\log_{\bx_i}\bx_j}{d(\bx_i, \bx_j)}, \\
\displaystyle (\bx_i, \bv_i)(0) =  (\bx^0_i, \bv^0_i) \in T {\mathcal M},
\end{cases}
\end{equation}
where $d(\bx_i, \bx_j) := \mbox{dist}(\bx_i, \bx_j)$ denotes the length minimizing geodesic distance between $\bx_i$ and $\bx_j$.  A formal derivation of the bonding force in $\eqref{A-8}_2$ in a Riemannian setting will be sketched in  Appendix \ref{App-E}.
Similar to Euclidean RCS model  \eqref{A-7}, we first provide an elementary energy estimate (Proposition \ref{P5.2}) and using this energy estimate, we provide a simple analytical framework in terms of system parameters and initial data for collision avoidance (Corollary \ref{C5.1}) and a global well-posedness of \eqref{A-8} (Theorem \ref{T5.1}).  For $\kappa_1 = \kappa_2 = 0$, the emergent dynamics of \eqref{A-8} was already studied in \cite{A-H-K-S}. However, unlike to the Euclidean case, emergent dynamics was verified under a priori condition on the spatial positions which cannot be justified using the initial data and system parameters. Thanks to the bonding force field with $\kappa_2 > 0$, we can derive a uniform bound for the relative distances so that we do not assume any a priori bound for spatial positions. Instead, we assume a small initial relative distances bounded by the injectivity radius of the ambient manifold (see Theorem \ref{T5.1}). 

\begin{figure}
\centering
\begin{minipage}{0.3\textwidth}
\begin{subfigure}[c]{2 in}
\epsfig{file=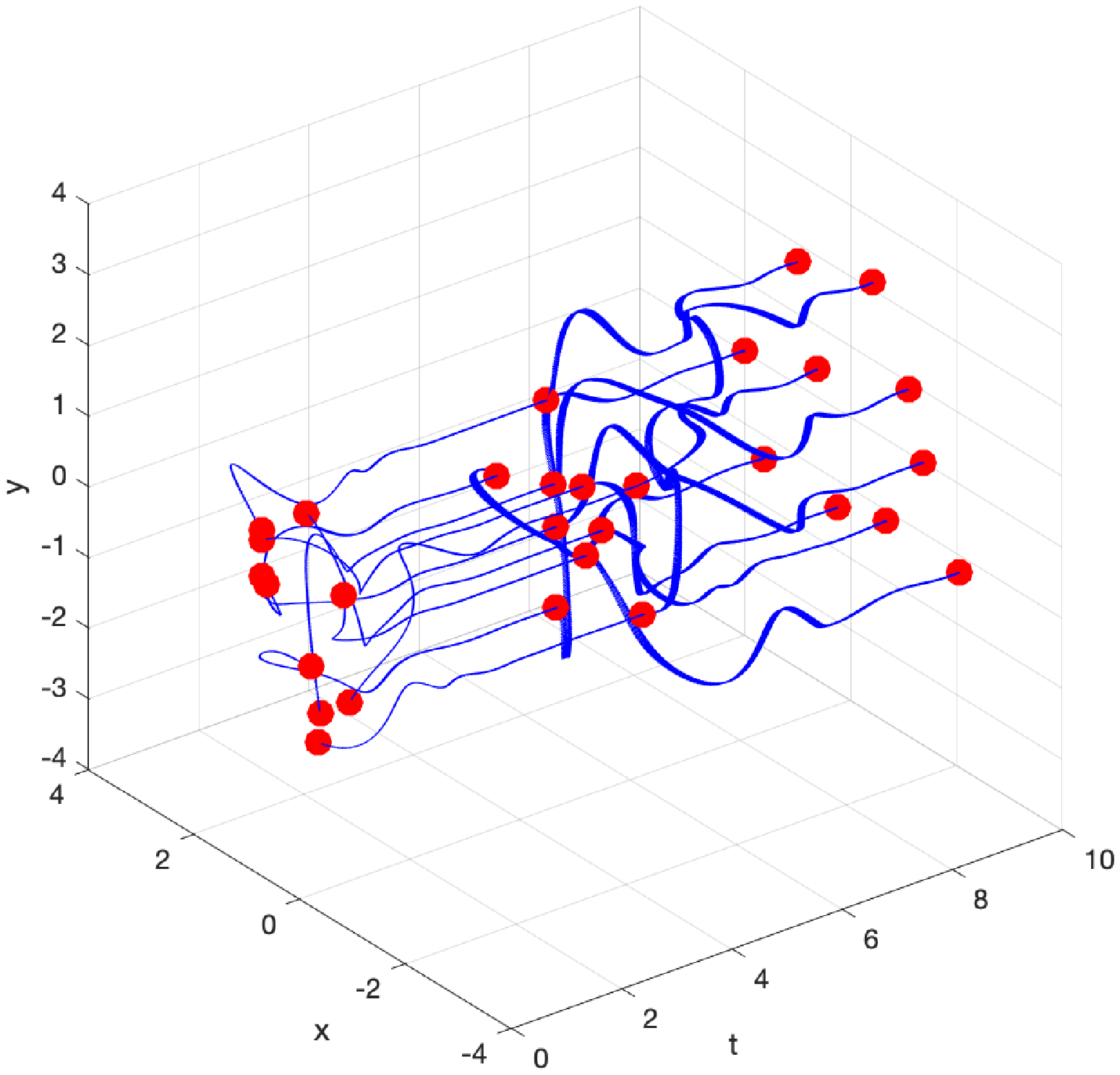, angle=0, width=5cm}
\caption{Worldlines of CSBF}
\end{subfigure}
\end{minipage}\hfill
\begin{minipage}{0.3\textwidth}
\begin{subfigure}[c]{2.3 in}
\epsfig{file=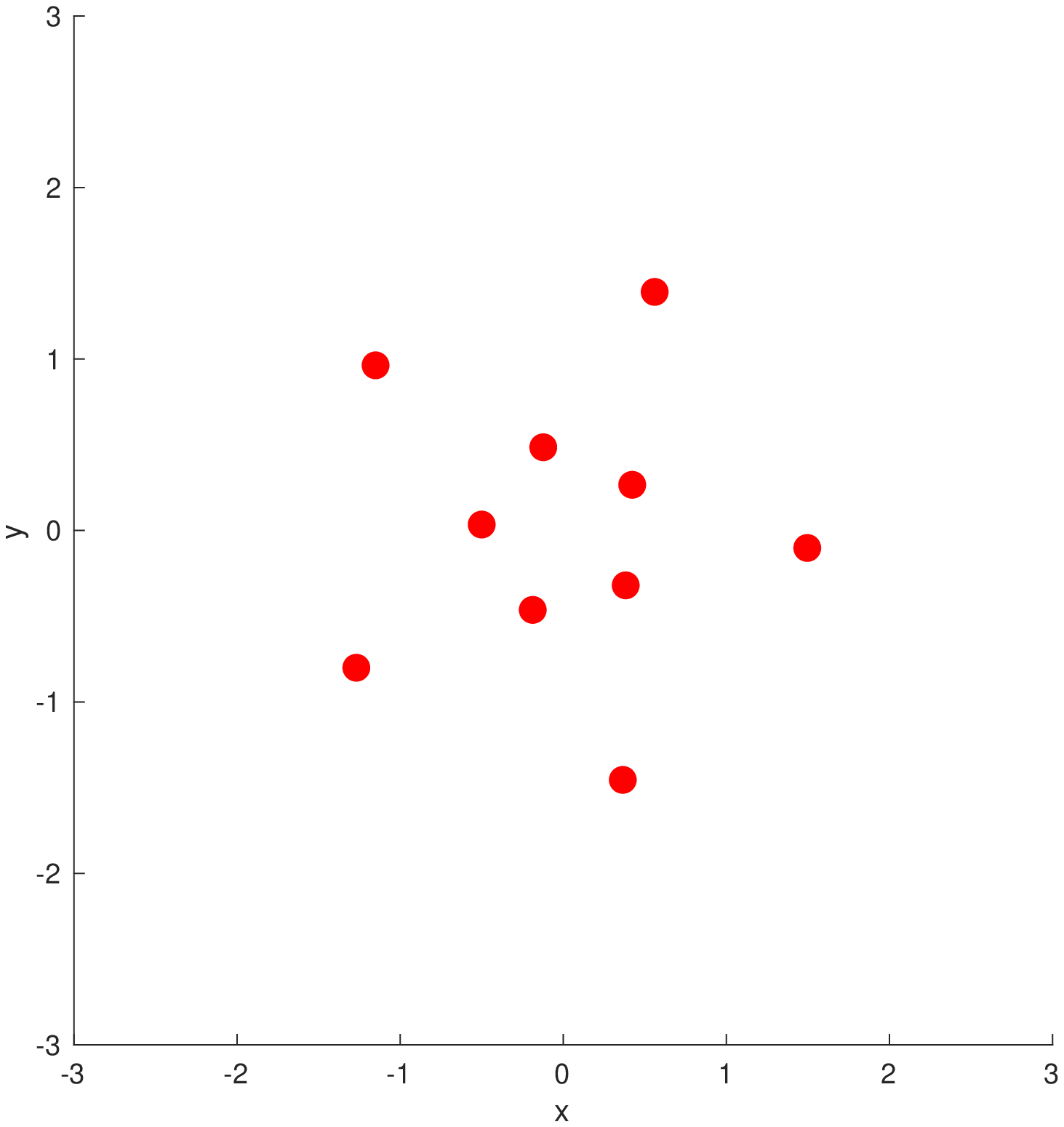, angle=0, width=5cm}
\caption{Star-shaped pattern}
\end{subfigure}
\end{minipage}\hfill
\begin{minipage}{0.3\textwidth}
\begin{subfigure}[c]{2.5 in}
\epsfig{file=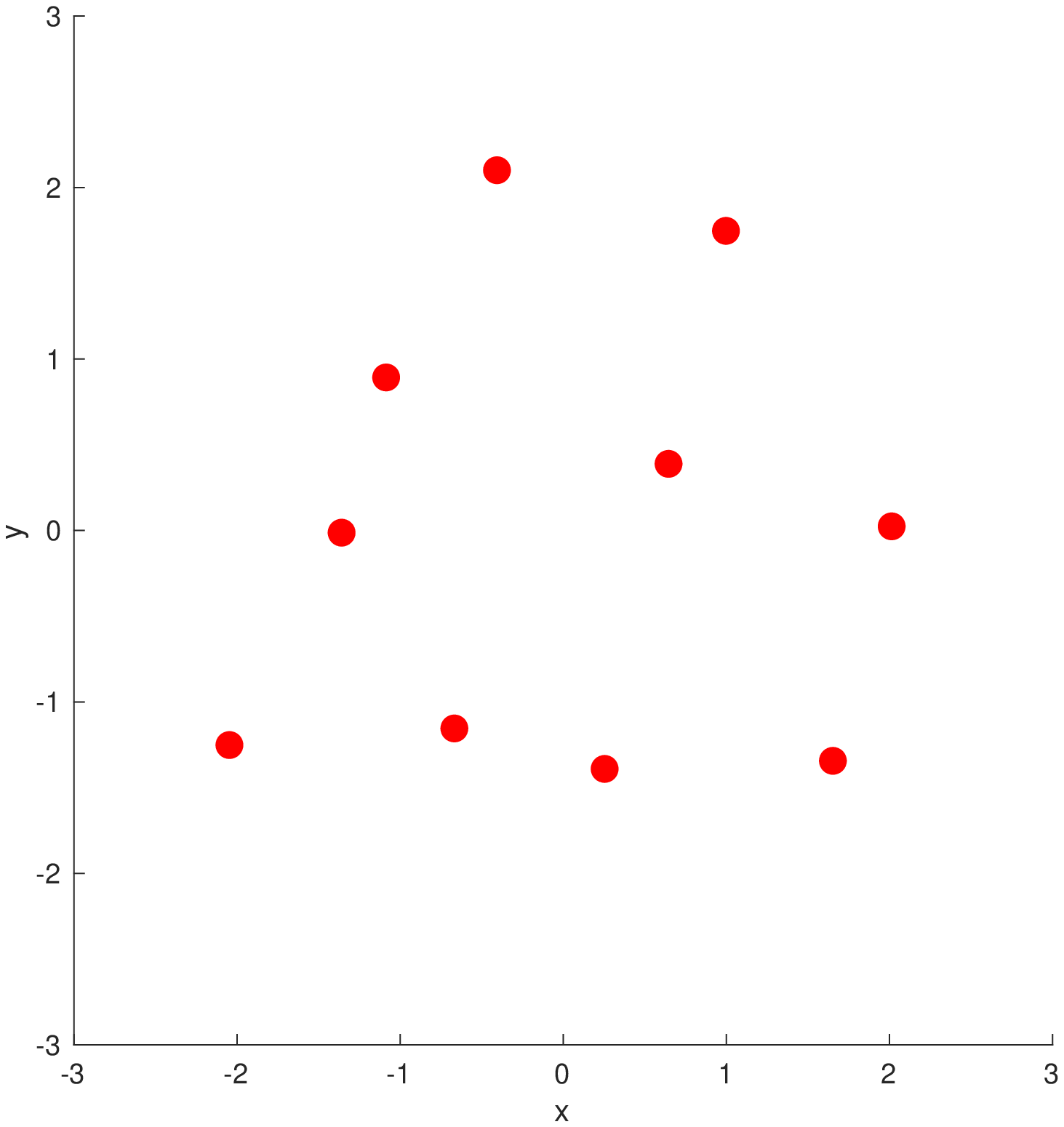, angle=0, width=5cm}
\caption{Heart-shaped pattern}
\end{subfigure}
\end{minipage}
\caption{Emergence of specified patterns}
\label{CSfig}
\end{figure} \vspace{.2cm}
 
The rest of this paper is organized as follows. In Section \ref{sec:2}, we briefly present basic estimates on maximal speed and elementary energy estimate for the RCS model \eqref{A-7}. In Section \ref{sec:3}, we study collision avoidance and asymptotic flocking of the RCS model.  In section \ref{sec:4}, we show that the relativistic model reduces to the classical CS model in any finite-time interval, as the speed of light tends to infinity. In section \ref{sec:5}, we study collision avoidance and global well-posedness of the manifold RCS model with bonding force. Finally, Section \ref{sec:6} is devoted to the brief summary of our main results and some discussion on the remaining issues to be investigated in a future work. In Appendix \ref{App-A}, we provide an explicit example for finite-time collisions using the CS model with a bonding force. In Appendix \ref{App-B} and Appendix \ref{App-C},  we present a proof of Lemma \ref{L4.1}, and derivation of Gronwall's inequality for the deviation functions from the RCS solution to the CS solution. In Appendix \ref{App-D}, we present a proof of Lemma \ref{L5.1}. Finally, Appendix \ref{App-E} is devoted to the heuristic derivation of bonding force on manifolds.  \newline

\noindent \textbf{Notation}: For simplicity, we set
\begin{align}
\begin{aligned} \label{A-9}
&  \bv_i := \hat{v}(\bw_i), \quad  \bv_j := \hat{v}(\bw_j), \quad \bx_{ij}:= \bx_i - \bx_j, \\
&  \bv_{ij} := \bv_i - \bv_j, \quad  \bw_{ij} := \bw_i - \bw_j, \quad  r_{ij} := | \bx_{ij} |, 
\end{aligned}
\end{align}
and we also use the following handy notation from time to time:
\[
\max_{i} := \max_{1 \leq i \leq N}, \quad
\max_{i,j} := \max_{1 \leq i,j \leq N}, \quad
\max_{i \neq j} := \max_{\substack{1 \leq i,j \leq N \\ i \neq j}}, 
\]
and the same things can be applied to $\min$ as well. The open ball of radius $r$ with center $\bx$ will be denoted by
\[
B_r(\bx):=\{ ~ \by \in \bbr^d: ~  | \by - \bx | < r ~ \}.
\]

\section{Preliminaries} \label{sec:2}
\setcounter{equation}{0}
In this section, we present the RCS model with a bonding force on the Euclidean space $\bbr^d$, and we provide basic estimates such as maximal speed and energy estimates which will be crucial in later sections. 

\subsection{The RCS model with a bonding force on $\bbr^d$} \label{sec:2.1}
Let $\bx_i \in \bbr^d, ~ \bv_i \in B_c(\bf{0})$ and $\bw_i \in \bbr^d$ be the position, the velocity and momentum-like quantity of the $i$-th CS particle defined in \eqref{A-0}. \newline

Recall the momentum equation:
\begin{equation}
\begin{cases}\label{B-1}
\displaystyle \dot\bw_i  = \frac{\kappa_0}{N}\sum_{j=1}^{N}\phi(|\bx_i-\bx_j|)\left(\hat{v}(\bw_j) - \hat{v}(\bw_i) \right) \\
\vspace{0.1cm}
\displaystyle \hspace{0.5cm} +~\frac{1}{2N} \sum_{\substack{j=1 \\ j \neq  i }}^{N} 
	\Big[ \kappa_1 \Big \langle  \hat{v}(\bw_j)- \hat{v}(\bw_i), \frac{\bx_j-\bx_i}{| \bx_j - \bx_i|} \Big \rangle   + \kappa_2 \Big(|\bx_i-\bx_j|- R^{\infty}_{ij} \Big) \Big]  \frac{(\bx_j-\bx_i)}{| \bx_j - \bx_i|}.
\end{cases}
\end{equation}
Although the R.H.S. of $\eqref{B-1}$ looks complicated, it is easy to see that they are skew-symmetric with respect to index exchange transformation $(i, j) ~ \leftrightarrow ~ (j, i)$. This leads to the conservation of total sum of $\bw_i$.
\begin{lemma} \label{L2.1}
Let $\{(\bx_i,\bw_i)\}$ be a solution to \eqref{A-7} with the initial data $\{ (\bx_i^0, \bw_i^0) \}$.  Then, the total sum of $\bw_i$ is conserved: 
\[  \sum_{i=1}^{N} \bw_i(t) =  \sum_{i=1}^{N} \bw_i^0, \quad t \geq 0.  \]
\end{lemma}
\begin{proof} We sum up $\eqref{B-1}_2$ with respect to $i$ to get 
\[ \frac{d}{dt} \sum_{i=1}^{N} \bw_i  = \sum_{i=1}^{N}  \frac{d\bw_i}{dt} = \sum_{i=1}^{N} \mbox{R.H.S. of $\eqref{A-7}_2$} = {\bf 0}. \]
This leads to the desired estimate. 
\end{proof}

 In \cite{A-H-K, H-K-R}, the authors introduced the relativistic kinetic energy $\mathcal{E}^c_k$ corresponding to the classical kinetic energy for the Cucker-Smale model in the nonrelativistic limit $(c \to \infty)$. In addition, we introduce the potential energy $\mathcal{E}_p^c$ arised from the bonding interactions.

\begin{definition}
\label{D2.1} \emph{(Relativistic energies) \cite{A-H-K,H-K-R}}
	Let $\{(\bx_i,\bw_i)\}$ be a solution to \eqref{A-7}. Then, the associated kinetic and potential energies are defined as follows.
	\begin{enumerate}
	\item Relativistic kinetic energy:
	\begin{equation} \label{B-1-1}
		\mathcal{E}^c_k
		:=\sum_{i=1}^{N} \left( c^2(\Gamma_i-1)+(\Gamma_i^2-\log{\Gamma_i}) \right),
	\end{equation}
	where $\Gamma_i := \Gamma(\bv_i)$. 
	\vspace{0.1cm}
	\item Potential energy:
	\begin{equation} \label{B-1-2}
		\mathcal{E}^c_p:=\frac{\kappa_2}{8N} \sum_{\substack{i,j =1 \\ i \neq j}}^{N} \Big (|\bx_i-\bx_j|-R^\infty_{ij} \Big )^2.
	\end{equation}
	\item Relativistic total energy:
	\begin{equation*} \label{B-2}
		\mathcal{E}^c:=\mathcal{E}^c_k+\mathcal{E}^c_p.
	\end{equation*}
	\end{enumerate}
\end{definition}
\begin{remark}\label{R2.1} 
Below, we provide several comments on the relativistic kinetic energy and potential energy. 

\begin{enumerate}
\item
Note that  the kinetic and potential energies  depend on the speed of light $c$ via the Lorentz force explicitly or implicitly. So to emphasize this $c$-dependence, we use $c$ as superscript for kinetic, potential and total energies.
\item 
By Taylor's expansion, one can see
\begin{equation} \label{B-2-1}
 \Gamma_i =  \frac{1}{\sqrt{1- \frac{| \bv_i |^2}{c^2}}} = 1 + \frac{|\bv_i|^2}{2c^2} + \frac{3 |\bv_i |^4}{8c^4} + \cdots.
 \end{equation}
It is easy to check that  $\Gamma_i$ is a function of $\bv_i$ and $c$, and  it satisfies 
\[ \frac{\partial \Gamma_i}{\partial c} \leq 0, \quad  \lim_{c \to \infty} \Gamma_i = 1 \quad \mbox{and} \quad \lim_{c \to \infty} c^2(\Gamma_i-1)=\frac{|\bv_i|^2}{2}. \]
Thus, as $c \to \infty$, the relativistic kinetic energy \eqref{B-1-1} tends to the translation of the classical mechanical kinetic energy:
\[
	\lim_{c\rightarrow \infty}\mathcal{E}^c_k(t) = \sum_{i=1}^{N}\frac{ |\bv_i|^2}{2} + N.
\]
The Lorentz factor $\Gamma_i$ can be viewed as a function of $c$ and $t$ along a solution $(\bx_i, \bw_i)$  to \eqref{A-7}. Since 
 \[
\frac{\partial}{\partial c} (\Gamma_i^2-\log{\Gamma_i}) = \Big( 2 \Gamma_i - \frac{1}{\Gamma_i} \Big) \frac{\partial \Gamma_i}{\partial c} = \Big( \frac{2 \Gamma_i^2 - 1}{\Gamma_i} \Big) \cdot \frac{\partial \Gamma_i}{\partial c} \leq 0, 
\]
one has 
\[
\Gamma_i(t,c) ^2-\log{\Gamma_i(t,c)} \geq \lim_{c \to \infty} \Gamma_i^2-\log{\Gamma_i} = 1, \quad t \geq 0.
\]
In particular, this implies 
\begin{equation} \label{B-2-2}
{\mathcal E}_k^c(t) \geq N.
\end{equation}
\vspace{0.2cm}
	\item The potential energy $\mathcal{E}^c_p$ can be understood as a quantity measuring how much particles are deviated from the desired spatial pattern registered  by $\mathcal{R}^\infty=[R^\infty_{ij}]_{i,j=1}^N$. In terms of the Frobenius norm,
	 it can also be rewritten as
	\[ \mathcal{E}^c_p(t) = \frac{\kappa_2}{8N}\| \mathcal{R}(t) - \mathcal{R}^{\infty} \|_F^2,	\]
	where $N \times N$ matrix ${\mathcal R}$ is given as 
	\[ r_{ij} (t) = | \bx_i(t) - \bx_j(t)|, \quad {\mathcal R} := [r_{ij}]_{i,j=1}^N.\]
	\end{enumerate}
\end{remark}

\subsection{Maximal speed and energy estimate} \label{sec:2.2}
In this subsection, we study estimates on the maximal speed and  time-evolution of total energy. Before we study time-evolution of energy, we first recall useful comparability relation between relative velocity and relative momentum.
\begin{lemma}\label{L2.2} 
\emph{\cite{B-H-K}}
For $\tau \in (0, \infty)$, let $\{(\bx_i,\bw_i)\}$ be a solution to \eqref{A-7} in the time interval $[0,\tau)$ satisfying a priori condition:~there exists a positive constant $U_w$ such that
	\begin{equation*} \label{B-2-3}
		\sup_{0 \leq t < \tau}\max_{1 \leq i \leq N} |\bw_i(t)| \leq U_w < \infty. 
	\end{equation*}
	Then, there exists a positive constant $C_L < 1$ such that 
\begin{align*}
\begin{aligned}
&(i)~C_{L} |\bw_i-\bw_j | \le |\bv_i-\bv_j | \le |\bw_i-\bw_j|, \\
&(ii)~C_{L}^2  |\bw_i-\bw_j|^2 \le \langle \bw_i-\bw_j , \bv_i-\bv_j \rangle \le |\bw_i-\bw_j|^2,
\end{aligned}
\end{align*}
where $\langle \cdot, \cdot \rangle$ is the standard $\ell^2$-inner product in $\bbr^d$ and  $C_{L}$ is a positive constant depending only on $U_w$ and $c$,  and it satisfies asymptotic relation:
\[	 \lim_{c \to \infty}C_L=1 \quad \mbox{ for each $U_w$}. \] 
\end{lemma}
\begin{proof}
	Since the proof is basically the same as in \cite[Lemma 2.2]{B-H-K}, we omit its detail. Note that in the proof of \cite[Lemma 2.2]{B-H-K}, $U_w$ was taken by $\max_{i} |\bw_i^0|$, because $\max_{i}|\bw_i(t)|$ decreases in time. 
\end{proof}
\begin{remark} \label{R2.2}
Next, we provide several comments on the result of Lemma \ref{L2.2}. 
\begin{enumerate}
\item
In the proof of \emph{\cite[Lemma 2.2]{B-H-K}},  we can see that the eigenvalues of the Jacobian for the mapping $\hat{w}:\bv \mapsto \bw $, counting multiplicity, can be calculated explicitly:
	\[
		 \lambda_1=\cdots=\lambda_{d-1}=\Gamma+\frac{\Gamma^2}{c^2}, 
		 \qquad \lambda_d = \lambda_1 + \frac{\Gamma(\Gamma^2-1)(c+2\Gamma)}{c^2}.
	\]
This and \eqref{B-2-1} imply
	\[
		1-\frac{1}{\lambda_i} = \mathcal{O}(c^{-2}), \quad \mbox{for}~i \in [d].
	\]
	On the other hand, by the mean value theorem we have
		\begin{equation}\label{B-3}
			| \bw - \bv | \leq \| \mathrm{Id} - \hat{v} \|_{op} |\bw| 
			= \left(1 - \frac{1}{\lambda_d}\right) |\bw| \leq U_w \mathcal{O}(c^{-2}),
		\end{equation}
	where $\| \cdot \|_{op}$ stands for the operator norm. 
\vspace{0.2cm}
\item
If we assume 
\begin{equation} \label{B-3-1}
\sum_{i=1}^N {\bw}_i^0 = {\bf 0},
\end{equation}
then it follows from Lemma \ref{L2.1} that 
\[
\bw_c(t) := \frac{1}{N} \sum_{i=1}^N {\bw}_i(t) = {\bf 0}, \quad  t \geq 0.
\] 
Therefore, the estimates in  Lemma \ref{L2.2} and zero sum condition \eqref{B-3-1} imply
\begin{align}
\begin{aligned} \label{B-3-2}
&\lim_{t\to\infty}\max_{i,j} |\bv_i-\bv_j|=0 \\
&\hspace{1cm} \iff \lim_{t\to\infty}\max_{i,j} |\bw_i-\bw_j|=0 \\
&\hspace{1cm} \iff \lim_{t\to\infty}\max_{i} |\bw_i|=0 \iff \lim_{t\to\infty}\max_{i} |\bv_i|=0,
\end{aligned}
\end{align}
Thanks to \eqref{B-3-2}, aggregation of velocities formulated in terms of $|\bv_j- \bv_i|$ can recast as the corresponding relations for $|\bw_j - \bw_i|$ as well. We also note that unlike to the relativistic momentum variables, the sum of relativistic velocity is not conserved along \eqref{A-7}.
\vspace{0.2cm}
\item
System \eqref{B-1} lacks the Galilean invariance \emph{(cf. \cite[Lemma 2.2]{A-B-H-Y})}. Nevertheless, thanks to the nonrelativistic limit (Theorem \ref{T4.1}), we can see that the solution to  \eqref{B-1} can be approximated by the solution of the nonrelativistic model in \emph{\cite{A-B-H-Y}}, which is Galilean invariant.
\end{enumerate}
\end{remark}

\vspace{0.5cm}

In next proposition, we study the time-evolution of relativistic energy ${\mathcal E}^c$ introduced in Definition \ref{D2.1}. 
\begin{proposition}\label{P2.1}
\emph{(Energy estimate)}
For $\tau  \in (0,\infty]$, let $\{(\bx_i,\bw_i)\}$ be a solution to \eqref{A-7} in the time-interval $[0,\tau)$. Then, $\mathcal{E}^c$ satisfies
\begin{equation*} \label{B-3-3}
	\mathcal{E}^c(t)+\int_{0}^{t} {\mathcal P}^c(s)ds=\mathcal{E}^c(0), \quad t \geq 0,
\end{equation*}
where ${\mathcal P}^c$ is the total energy production functional:
\[
	{\mathcal P}^c(t):=
	\frac{\kappa_0}{2N}\sum_{i,j=1}^{N}\phi(|\bx_j - \bx_i |) | \bv_j - \bv_i |^2
	+ \frac{\kappa_1}{4N} \sum_{\substack{i,j =1 \\ j \neq i}}^{N}  \Big \langle \bv_j - \bv_i, \frac{\bx_j - \bx_i}{|\bx_j - \bx_i|}  \Big \rangle^2.
\]
\end{proposition}
\begin{proof}
It follows from $\eqref{A-7}_{2}$, \eqref{A-9} and  \eqref{B-3-3} that
\begin{align}
\begin{aligned} \label{B-4}
\langle \bv_i , \dot{\bw}_i \rangle
=& \frac{\kappa_0}{N}\sum_{j=1}^{N}\phi(r_{ji})\big\langle  \bv_{ji},\bv_i\big\rangle + \frac{\kappa_1}{2N} \sum_{\substack{j =1 \\ j \neq i}}^{N}
 \Big \langle \bv_{ji}, \frac{\bx_{ji}}{r_{ji}}  \Big \rangle \cdot \Big \langle \frac{\bx_{ji}}{r_{ji}},\bv_i \Big\rangle \\
&+ \frac{\kappa_2}{2N} \sum_{\substack{j =1 \\ j \neq i}}^{N} (r_{ji}- R^\infty_{ij} ) \Big \langle \frac{\bx_{ji}}{r_{ji}},\bv_i \Big \rangle.
\end{aligned}
\end{align}
We add \eqref{B-4} over all $i \in [N]$, and then perform the index switching trick $i \leftrightarrow j$ for the resulting relation to obtain
\begin{equation} \label{B-5}
\sum_{i=1}^N \langle \bv_i , \dot{\bw}_i \rangle
= -\frac{\kappa_0}{2N}\sum_{i,j=1}^{N}\phi(r_{ji}) | \bv_{ji} |^2 -\frac{\kappa_1}{4N}  \sum_{\substack{i,j =1 \\ j \neq i}}^{N} \Big \langle \bv_{ji}, \frac{\bx_{ji}}{r_{ji}} \Big \rangle^2
-\frac{\kappa_2}{4N}  \sum_{\substack{i,j =1 \\ j \neq i}}^{N} (r_{ji}-R^\infty_{ij})  \Big \langle \frac{\bx_{ji}}{r_{ji}},\bv_{ji} \Big \rangle.
\end{equation}
Next, we claim the following relations:
\begin{equation} \label{B-6}
 \frac{d\mathcal{E}^c_k}{dt}: =\sum_{i=1}^N \langle \bv_i , \dot{\bw}_i \rangle, \quad   \frac{d\mathcal{E}^c_p}{dt}:= \frac{\kappa_2}{4N}  \sum_{\substack{i,j =1 \\ j \neq i}}^{N} 
 (r_{ij}-R^\infty_{ij})\Big \langle \frac{\bx_{ji}}{r_{ji}},\bv_{ji} \Big \rangle.
\end{equation}

\vspace{0.2cm}

\noindent $\bullet$~(Derivation of $\eqref{B-6}_1$):~We use the relation:
\[
(c^2 - |\bv_i|^2)\Gamma_i^2 = c^2 \]
to find 
\[ 2\langle \bv_i, \dot{\bv}_i \rangle = \frac{2c^2}{\Gamma_i^3}\frac{d\Gamma_i}{dt}. \]
This and \eqref{B-1-1} yield
\begin{align*}
\begin{aligned}
\sum_{i=1}^N \langle \bv_i , \dot{\bw}_i \rangle
=& \sum_{i=1}^N \left[ \frac{d}{dt} \langle \bv_i, \bw_i \rangle - \langle \dot{\bv}_i, \bw_i \rangle \right] 
= \sum_{i=1}^N \left[ \langle \dot{\bv}_i, \bv_i \rangle F_i + | \bv_i |^2 \dot{F}_i \right] \\
=&\sum_{i=1}^{N} \left[ \big\langle\dot{\bv}_i,\bv_i\big\rangle\bigg(\Gamma_i\bigg(1+\frac{\Gamma_i}{c^2}\bigg)\bigg)+ |\bv_i|^2\frac{d}{dt}\bigg(\Gamma_i\bigg(1+\frac{\Gamma_i}{c^2}\bigg)\bigg) \right] \\
=&\sum_{i=1}^{N}\bigg[\frac{c^2}{\Gamma_i^2}\bigg(1+\frac{\Gamma_i}{c^2}\bigg)+ |\bv_i|^2\bigg(1+\frac{2\Gamma_i}{c^2}\bigg)\bigg]\frac{d\Gamma_i}{dt}\\
=&\sum_{i=1}^{N}\bigg[c^2+\Gamma_i\bigg(1+\displaystyle\frac{|\bv_i|^2}{c^2}\bigg)\bigg]\frac{d\Gamma_i}{dt}
=\sum_{i=1}^{N}\bigg(c^2+\Gamma_i\bigg(2-\frac{1}{\Gamma_i^2}\bigg)\bigg)\frac{d\Gamma_i}{dt}\\
=&\sum_{i=1}^{N}\frac{d}{dt}\bigg[ c^2\bigg(\Gamma_i-1\bigg)+\bigg(\Gamma_i^2-\log{\Gamma_i}\bigg)\bigg ]= \frac{d}{dt} \mathcal{E}^c_k(t).
\end{aligned}
\end{align*}

\vspace{0.2cm}

\noindent $\bullet$~(Derivation of $\eqref{B-6}_2$):~We use \eqref{B-1-2} to see
\[
\frac{d\mathcal{E}^c_p(t)}{dt}= \frac{\kappa_2}{8N} \sum_{\substack{i,j =1 \\ i \neq j}}^{N} \frac{d}{dt} \Big ( r_{ij} -R^\infty_{ij} \Big )^2 
= \frac{\kappa_2}{4N} \sum_{i,j=1}^{N} (r_{ij}-R^\infty_{ij}) \Big \langle \frac{\bx_{ji}}{r_{ji}},\bv_{ji} \Big \rangle.
\]
Finally, we combine \eqref{B-5} and \eqref{B-6} to obtain
\[
\frac{d}{dt}\mathcal{E}^c(t)=\frac{d}{dt}\mathcal{E}^c_k(t) + \frac{d}{dt}\mathcal{E}^c_p(t) =-\frac{\kappa_0}{2N}\sum_{i,j=1}^{N}\phi(r_{ji}) | \bv_{ji} |^2-\displaystyle\frac{\kappa_1}{4N}\sum_{i,j=1}^{N}  \Big \langle \bv_{ji} , \frac{\bx_{ji}}{r_{ji}} \Big \rangle^2 =-{\mathcal P}^c(t).
\]


\end{proof}

\begin{remark}The total energy production can also be written as
	\[
		{\mathcal P}^c(t)=\frac{\kappa_0}{2N}\sum_{i,j=1}^{N}\phi(r_{ji}) | \bv_{ji} |^2+\displaystyle\frac{\kappa_1}{4N}\sum_{i,j=1}^{N}\bigg(\frac{dr_{ji}}{dt}\bigg)^2 \geq 0.
	\]
\end{remark}
Before we close this section, we recall Barbalat's lemma to be used in later sections.
 \begin{lemma}\label{Barbalat}\emph{(Barbalat's lemma)}
 	Suppose $f:(0,\infty)\to\mathbb{R}$ is a uniformly continuous function such that 
	\[ \int_0^{\infty} f(t) dt < \infty. \]
	Then, one has 
	\[ \lim_{t \to \infty} f(t) = 0. \]
 \end{lemma}

\section{Collision avoidance and asymptotic flocking} \label{sec:3}
\setcounter{equation}{0}
In this section, we study a sufficient framework for the global well-posedness and the asymptotic flocking of \eqref{A-7}.

\subsection{Collision avoidance} \label{sec:3.1}
In this subsection, we first provide a sufficient framework which guarantees collision avoidance. Unlike to the RCS model without bonding force, there are two subtle issues concerning a global well-posedness: 
\vspace{0.1cm}
\begin{itemize}
\item
Due to the bonding forcing terms in the R.H.S. of $\eqref{A-7}_2$, we can not directly deduce that maximal speed can not reach the speed of light. Since $|\bv| \to c-$ implies a blow-up of momentum $|\bw| \to \infty$, therefore, one needs to show that maximal speed cannot reach $c$ in any finite time to guarantee the global well-posedness of \eqref{A-7}.
 \vspace{0.2cm}
 \item
Finite-time collisions of particles can make the R.H.S. of \eqref{B-1}  be discontinuous due to the presence of ${\bx_{ji}}/{r_{ji}}$ in \eqref{A-7}. Thus, we study sufficient conditions to resolve such subtle issue and verify the unique existence of global-in-time solutions via the energy estimate in Proposition \ref{P2.1}.
\end{itemize}

\vspace{0.2cm}

In next lemma, we derive lower and upper bounds for maximal speed and relative spatial positions. For this, we set 
\begin{equation} \label{C-0}
\underline{r} := \min_{k \neq l}R_{kl}^\infty-\sqrt{\frac{4N(\mathcal{E}^c(0)-N)}{\kappa_2}}, \qquad 
\overline{r} := \max_{k \neq l}R^\infty_{kl}+\sqrt{\frac{4N(\mathcal{E}^c(0)-N)}{\kappa_2}}.
\end{equation}
\begin{lemma}\label{L3.1}
For $\tau  \in (0,\infty]$, let $\{(\bx_i,\bw_i)\}$ be a solution to \eqref{A-7} in the time interval $[0,\tau)$. Then, the following assertions hold:
\begin{enumerate}
	\item If initial speeds do not exceed the speed of light, then the modulus of relativistic velocity is uniformly bounded away from the speed of light:
	\begin{align}\label{C-1}
		\max_i | \bv_i^0 | < c \quad  \Longrightarrow \quad \sup_{t \in [0,\tau)}\max_i | \bv_i(t) | < c.
	\end{align}
	Therefore, the existence of the $U_w$ and $C_L$ in Lemma \ref{L2.2} can be guaranteed.
	
	\vspace{0.2cm}
	
	\item The relative distances are bounded below and above as follows: for $i \neq j$, 
	\begin{align}\label{C-2}
	\underline{r} \leq | \bx_i(t) - \bx_j(t)| \leq \overline{r}, \quad \forall~t \in [0,\tau).
	\end{align}
	\item
	If the following condition 
	\begin{align}\label{C-3}
	\mathcal{E}^c(0)<N+\frac{\kappa_2}{4N}\min_{i \neq j}\left(R^\infty_{ij}\right)^2
	\end{align}
	holds, then collisions do not occur in the time interval $[0, \tau)$:
	\[ \inf_{t \in [0,\tau)} \min_{i \neq j} | \bx_i(t) - \bx_j(t) | > 0. \]
\end{enumerate}
\end{lemma}
\begin{proof} We use Proposition \ref{P2.1} to bound velocity and position via kinetic energy and potential energy, respectively. \newline

\noindent (1)~We use Definition \ref{D2.1}, Remark \ref{R2.1}, \eqref{B-2-1}, \eqref{B-2-2} and Proposition \ref{P2.1} to see
\begin{align}\label{C-4}
 \frac{1}{2} |\bv_i(t) |^2 + 1
	\leq c^2(\Gamma_i(t)-1)+(\Gamma_i^2(t)-\log{\Gamma_i(t)})
	\leq \mathcal{E}^c_k(t)
	\leq \mathcal{E}^c(t)
	\leq \mathcal{E}^c(0) < \infty,
\end{align}
for $i \in [N]$. 
On the other hand, it is easy to see the following implications:
\[ |\bv_i| \to c \quad \Longrightarrow \quad \Gamma_i \to \infty \quad \Longrightarrow \quad \mathcal{E}_k \to \infty. \]
Therefore if a speed tends to the speed of light for some particle $j$ and time $t^* \in [0,\tau)$, i.e.,
\[ \lim_{t \to t^*-} |\bv_j(t)|=c, \]
which contradicts \eqref{C-4}.  This verifies \eqref{C-1}.

\vspace{0.2cm}

\noindent (2)~Again, we use Proposition \ref{P2.1} to bound the potential energy as follows:
	\[
	{N+\mathcal{E}^c_p(t)\leq\mathcal{E}^c_k(t)+\mathcal{E}^c_p(t)
	}=\mathcal{E}^c_k(t)+\frac{\kappa_2}{8N}\sum_{i,j=1}^{N}(r_{ij}(t)-R^\infty_{ij})^2\leq \mathcal{E}^c(0).
	\]
For $i \neq j$, one has
	\begin{align*}
	\frac{\kappa_2}{4N}(r_{ij}(t)-R^\infty_{ij})^2
	\leq \frac{\kappa_2}{8N}\sum_{i,j=1}^{N}(r_{ij}(t)-R^\infty_{ij})^2
	\leq \mathcal{E}^c(0)-\mathcal{E}^c_k(t)
	\leq \mathcal{E}^c(0)-N.
	\end{align*}
This implies 
	\[
	R^\infty_{ij}-\sqrt{\frac{4N(\mathcal{E}^c(0)-N)}{\kappa_2}} 
	\leq r_{ij}(t)
	\leq R^\infty_{ij}+\sqrt{\frac{4N(\mathcal{E}^c(0)-N)}{\kappa_2}},
	\]
for arbitrary distinct particles $i$ and $j$. Hence, we have the desired estimate \eqref{C-2}. 
	
\vspace{0.2cm}

\noindent (3)~The inequality implies the first term of \eqref{C-2} is positive, which is equivalent to \eqref{C-3}. Thus, $r_{ij}(t)$ is nonzero and finite-time collisions do not occur.
\end{proof}

\begin{remark} If the initial energy is sufficiently large, and the relation \eqref{C-3} does not hold, then collision may occur. The possibility of collision is illustrated in Appendix \ref{App-A}.
\end{remark}

Lemma \ref{L3.1} resolves the issues which may cause the ill-posedness of \eqref{B-1}. The first statement says that although a particle's speed may increase, it cannot exceed the speed of light. The second statement indicates that if we choose initial data and predetermined parameter $R^\infty_{ij}$ in a suitable manner, collision does not occur and the continuity of system \eqref{B-1} is guaranteed. Therefore, by the standard Cauchy-Lipschitz theory, we obtain the global well-posedness, which we summarize as follows.
\begin{theorem} \label{T3.1}
Suppose the initial data and system parameters satisfy 
\[
	\mathcal{E}^c(0)<N+\frac{\kappa_2}{4N}\min_{i \neq j}\left(R^\infty_{ij}\right)^2.
\]
Then, there exists a unique global-in-time solution to \eqref{B-1}. 
\end{theorem}
\begin{proof} It follows from Lemma \ref{L3.1} that finite-time collisions do not occur. Thus, we can apply the Cauchy-Lipschitz theory to derive a unique global solution. 
\end{proof}
\subsection{Asymptotic flocking} \label{sec:3.2} 
In this subsection, we provide estimates on the asymptotic flocking. First, we recall the concept of global flocking as follows.
\begin{definition}\label{D3.1}
	Let $Z=\{(\bx_i,\bv_i)\}$ be a global-in-time solution to \eqref{A-7}. 
	\begin{enumerate}
		\item 
		The configuration $Z$ exhibits (asymptotic) velocity alignment if 
		\begin{equation*}
		\lim_{t \to \infty} \max_{i,j} |\bv_j(t)- \bv_i(t)|=0.
		\end{equation*}
		\item
		The configuration $Z$ exhibits (asymptotic) flocking if
		\begin{align*}
		\begin{aligned} 
		\sup_{0 \leq t < \infty} \max_{i,j} |\bx_i(t)-\bx_j(t)|<\infty,\quad \lim_{t \to \infty} \max_{i,j} |\bv_j(t)-\bv_i(t)|=0.
		\end{aligned}
		\end{align*}
	\end{enumerate}
\end{definition}

\vspace{0.2cm}

Now we are ready to provide asymptotic flocking. 
\begin{theorem}\label{T3.2}
Suppose that communication weight and initial data with \eqref{PA-2} satisfy the following two conditions:
    \[ 0 < \min_{i \neq j} |\bx_i^0-\bx_j^0|,\quad \kappa_\ell > 0, \quad \ell = 0, 1, 2, \quad  \phi_m:=\min\left\{\phi(r):0\leq r\leq  \overline{r} \right \} >0, \]
and let $\{(\bx_i,\bw_i)\}$ be a global solution to \eqref{B-1} with $\sum_{i=1}^N\bw_i^0= {\bf 0}$. Then the asymptotic flocking emerges:
 \[ \sup_{0 \leq t < \infty} \max_{i,j} | \bx_i(t) - \bx_j(t) | \leq  \overline{r}, \qquad \lim_{t \to \infty} \max_{i,j} |\bv_j(t)-\bv_i(t)|=0. \]
\end{theorem}
\begin{proof}
\noindent (i)~The spatial boundedness follows from Lemma \ref{L3.1}. \newline

\noindent (ii)~It follows from Proposition \ref{P2.1} and Lemma \ref{L2.2} that 
	\begin{align*}
	\frac{d\mathcal{E}^c}{dt} \leq -\frac{\kappa_0}{2N}\sum_{i,j=1}^{N}\phi(r_{ji}) |\bv_{ji} |^2
	\leq -\frac{\kappa_0\phi_m}{2N}\sum_{i,j=1}^{N} |\bv_{ji} |^2
	\leq -\frac{\kappa_0C_L^2\phi_m}{2N}\sum_{i,j=1}^{N} |\bw_{ji} |^2.
	\end{align*}
	This yields
	\begin{equation*} \label{C-5-1}
		\int_{0}^{\infty} |\bw_{ji} (t)|^2dt
		\leq \int_{0}^{\infty}\sum_{i,j=1}^{N} |\bw_{ji} (t)|^2dt
		\leq \frac{2N\mathcal{E}^{c}(0)}{C_L^2\phi_m}< \infty.
	\end{equation*}
	To apply Lemma \ref{Barbalat}, we estimate the derivative of $|\bw_{ji}|$ as follows:
			\begin{align}
			\begin{aligned}\label{C-6}
			&\frac{1}{2} \left| \frac{d}{dt} |\bw_{ji}|^2 \right| \leq |\dot{{\bw}}_{ji}| \cdot |{\bw}_{ji}| \\
			& \le \sum_{k=1}^N \Bigg[\frac{\kappa_0 \phi_M}{N}\Big(|{\bv}_{kj} |+ |{\bv}_{ki} |\Big)
			+\frac{\kappa_1}{2N}\Big( |{\bv}_{kj}|+ |{\bv}_{ki} |\Big) +\frac{\kappa_2}{2N}\Big(\Big|r_{kj}-R_{kj}^\infty\Big|+\Big|r_{ki}-R_{ki}^\infty\Big|\Big)\Bigg] |{\bw}_{ji}|\\
			&\le \frac{C}{N} \sum_{k=1}^N \Bigg[ |{\bv}_{kj}|+ |{\bv}_{ki}|
			 + \Big|r_{kj}-R_{kj}^\infty\Big|+\Big|r_{ki}-R_{ki}^\infty\Big| \Bigg] |{\bw}_{ji} |\\
			 &\le \frac{C}{\sqrt{N}} \Bigg[ \left( \sum_{k=1}^N |{\bv}_{kj} |^2 \right)^\frac{1}{2} + \left( \sum_{k=1}^N  |{\bv}_{ki} |^2 \right)^\frac{1}{2} + \left( \sum_{k=1}^N \Big|r_{kj}-R_{kj}^\infty\Big|^2 \right)^\frac{1}{2} + \left( \sum_{k=1}^N \Big|r_{ki}-R_{ki}^\infty\Big|^2 \right)^\frac{1}{2} \Bigg] |{\bw}_{ji}|,
		\end{aligned}\end{align}
		where $C$ is a positive constant defined by 
		\[ C := \max\left\{ \kappa_0\phi_M + \frac{\kappa_1}{2}, \frac{\kappa_2}{2} \right\}. \]
		The second inequality in \eqref{C-6} is due to the fact that $\phi\leq \phi_M$. \newline
		
		\noindent To bound the R.H.S. of \eqref{C-6}, we first use Lemma \ref{L3.1} to bound $|{\bw}_{ji}|$:
		\begin{equation} \label{C-9}
		|{\bw}_j-{\bw}_i | \leq 2U_w < \infty.
		\end{equation}
				
		\noindent $\bullet$~(First two terms in the R.H.S. of \eqref{C-6}):~Note that 
		\begin{equation}  \label{C-7}
			\sum_{k=1}^N  |{\bv}_{kj}|^2 \leq 
			2\sum_{k=1}^N ( | \bv_k |^2 + | \bv_j |^2 ) \leq 4(N+1) \sum_{k=1}^N \left( \frac{1}{2} | \bv_k |^2 + 1 \right)
			\le 4(N+1)\mathcal{E}^c_k.
		\end{equation}
		
		\noindent $\bullet$~(Last two terms in the R.H.S. of \eqref{C-6}):~ We use the potential energy to obtain an upper bound as follows.
		\begin{align}\label{C-8}
			\sum_{k=1}^N \Big|r_{kj}-R_{kj}^\infty\Big|^2
			\leq \sum_{k,j=1}^N \Big|r_{kj}- R_{kj}^\infty\Big|^2
			= \frac{8N}{\kappa_2}\mathcal{E}^c_p.
		\end{align}
		In \eqref{C-6}, we combine \eqref{C-9}, \eqref{C-7} and \eqref{C-8} to see that $| \bw_j - \bw_i |^2$ is uniformly continuous.  Again, we use Lemma \ref{Barbalat} to get the desired convergence result. 
	\[
	\lim_{t \to \infty} |\bw_j - \bw_i |^2 = 0.
	\] 
We use the equivalence relation between relativistic momentum and relativistic velocity (Remark \ref{R2.2}) to complete the proof.
\end{proof}

\section{Nonrelativistic limit}  \label{sec:4}
\setcounter{equation}{0}
In this section, we study the rigorous justification of the nonrelativistic limit from the Cauchy problem to the relativistic RCS model:
\begin{equation}\label{D-1}
\begin{cases}
\displaystyle \dot{\bx}^c_i = \bv_i^c,\quad t>0,\quad  i \in [N],\\
\displaystyle \dot{\bw}^c_i=\displaystyle\frac{\kappa_0}{N}\sum_{j=1}^{N}\phi(r_{ji}^c) \bv_{ji}^c + \frac{\kappa_1}{2N}\sum_{\substack{j=1\\j\ne i}}^{N} \Big \langle \bv_{ji}^c,\frac{\bx_{ji}^c}{r^c_{ji}} \Big \rangle  \frac{\bx_{ji}^c}{r_{ji}^c}  + \frac{\kappa_2}{2N}\sum_{\substack{j=1\\j\ne i}}^{N} ( r_{ji}^c-R^\infty_{ij}) \frac{\bx_{ji}^c}{r^c_{ji}} \\
\displaystyle \hspace{0.6cm} =: {\mathcal I}_i^c + {\mathcal J}_i^c + {\mathcal K}_i^c \\
\displaystyle (\bx_i^c(0),\bw_i^c(0))=(\bx_i^{0},\bw_i^{0})\in \mathbb{R}^{2d},
\end{cases}
\end{equation}
to the Cauchy problem to the classical CS model:
\begin{equation}\label{D-2}
\begin{cases}
\displaystyle \dot{\bx}^\infty_i =\bw_i^\infty,\quad t>0,\quad  i \in [N],\\
\displaystyle \dot{\bw}^\infty_i=\frac{\kappa_0}{N}\sum_{j=1}^{N}\phi(r_{ji}^\infty) \bw_{ji}^{\infty}  + \frac{\kappa_1}{2N}\sum_{\substack{j=1\\j\ne i}}^{N}  \Big \langle \bw_{ji}^\infty, \frac{\bx_{ji}^\infty}{r_{ji}^{\infty}}  \Big \rangle \frac{\bx_{ji}^\infty}{r_{ji}^\infty} + \frac{\kappa_2}{2N}\sum_{\substack{j=1\\j\ne i}}^{N} (r_{ji}^\infty-R^\infty_{ij}) \frac{\bx_{ji}^\infty} {r_{ji}^\infty} \\
\displaystyle \hspace{0.7cm} =: {\mathcal I}_i^{\infty} + {\mathcal J}_i^{\infty} + {\mathcal K}_i^{\infty} \\
\displaystyle (\bx_i^{\infty}(0),\bw_i^{\infty}(0))=(\bx_i^0,\bw_i^0)\in \mathbb{R}^{2d},
\end{cases}
\end{equation}
as $c \to \infty$. \newline

Note that superscripts in $\bw^0_i,\bx^0_i$ and $R_{ij}^\infty$ do not refer to the speed of light but time. We also denote the kinetic energies of \eqref{D-1} and \eqref{D-2} by $\mathcal{E}_k^c$ and $\mathcal{E}^\infty_k$, respectively, and potential energy will be denoted in the same way. 
We reveal the effect of $c$ in $F^c_i$:
\[
	F_i^c = \frac{c}{\sqrt{c^2-|\bv_i|^2}} + \frac{1}{c^2-|\bv_i|^2}. 
\]
Since we will observe the effect of $c$ for fixed $\{ \bw_i^0 \}_{i=1}^N$,  we represent the Lorentz factor in terms of momentum:
\[
	\Gamma(\bv_i^c) = \frac{c}{\sqrt{c^2-|\bv_i^c|^2}}
	=\frac{F_i^c c}{\sqrt{(F_i^c)^2c^2-|\bw^c_i|^2}}
	= {\tilde \Gamma}^{F_i^cc}(\bw^c_i),
	\quad \tilde{\Gamma}^{c'}(\bw) := \frac{c'}{\sqrt{(c')^2-|\bw|^2}},
\]
where we used $F_i^c \bv_i^c=\bw^c_i$. Similar to Definition \ref{D2.1},  we define the corresponding kinetic energy as
\begin{align}\label{D-4}
		\tilde{\mathcal{E}}_k^c 
		:=\sum_{i=1}^{N} \left[  c^2(\tilde\Gamma^c(\bw_i^c))-1)+((\tilde{\Gamma}^c(\bw_i^c))^2-\log{\tilde\Gamma^c(\bw_i^c)}) \right ].
\end{align}
Before we verify the nonrelativistic limit, we first consider an uniform-in-$c$ analogue of Lemma \ref{L3.1}. 

\begin{lemma}\label{L4.1}
Suppose that initial data satisfy
	\[
		\min_{i \neq j} |\bx_i^0-\bx_j^0| >0, \quad   \max_{i} |\bw_i^0| < c,
	\]
and for $\tau \in (0,\infty]$, let $\{(\bx_i^c,\bw_i^c)\}$ and $\{(\bx_i^\infty,\bw_i^\infty)\}$ be two solutions to the Cauchy problems \eqref{D-1} and \eqref{D-2}, respectively, defined on the finite-time interval $[0,\tau)$. For $c' \in [c,\infty]$, let $\{( \bx_i^{c'}, \bw_i^{c'} )\}$ be a solution to $\eqref{D-1}$ corresponding to $c'$.
Then, the following assertions hold.
\begin{enumerate}
	\item If solutions are defined on $[0,\tau)$ for each $c' \in [c,\infty]$, then there exists a positive constant  $\delta<1$ such that 
	\begin{align}\label{D-5}
		\sup_{\substack{c' \in [c,\infty) \\ t \in [0,\tau)}}\max_{i}\frac{|\bv^{c'}_i(t)|}{c'} < \delta < 1, \quad \text{where} \quad F_i^{c'}\bv_i^{c'}={\bw_i^{c'}}.
	\end{align}
	Consequently, in terms of momentum, we have the following uniform bound:
	\begin{align}\label{D-6}
		\sup_{\substack{c' \in [c,\infty] \\ t \in [0,\tau)}}\max_{i} |\bw^{c'}_i(t)| < U_w < \infty
		\quad \text{for some} \quad U_w \in \bbr. 
	\end{align}
	\item
	If initial data satisfy
	\begin{align}\label{D-7}
		\tilde{\mathcal{E}}_k^c(0) + \mathcal{E}^c_p(0)
		\leq N + \frac{\kappa_2}{4N}\min_{i \neq j}\left(R^\infty_{ij}\right)^2,
	\end{align}
	where $\tilde{\mathcal{E}}_k^c$ is the kinetic energy defined in \eqref{D-4}, then collisions do not occur in finite time for arbitrary $c' \in [c,\infty]$. In particular, the solution is globally well-posed(i.e. $\tau=+\infty$). 
\end{enumerate}
\end{lemma}
\begin{proof} Since proofs are rather lengthy, we leave them in Appendix \ref{App-B}. 
\end{proof}
Now, we are ready to present the nonrelativistic limit from the relativistic model \eqref{D-1} to the classical model \eqref{D-2}.
\begin{theorem}  \label{T4.1}
\emph{(Finite-in-time nonrelativistic limit)} 
Suppose that the initial data satisfy
	\begin{align*}
	 \min_{i \neq j} |\bx_i^0-\bx_j^0| >0, \quad \max_{i} |\bw_i^0| < c, \quad  \tilde{\mathcal{E}}^c_k(0) + \mathcal{E}^c_p(0) \leq N + \frac{\kappa_2}{4N}\min_{i,j}\left(R^\infty_{ij}\right)^2,
	\end{align*}
and $\{(\bx_i^c,\bw_i^c)\}$ and $\{(\bx_i^\infty,\bw_i^\infty)\}$ are solutions to \eqref{D-1} and \eqref{D-2}, respectively. Then, for $T \in (0, \infty)$, one has 
\begin{equation*} \label{D-7-1}
	\lim_{c \to \infty}\sup_{0 \leq t \leq T}  \sum_{i=1}^N \Big(  |\bx_i^c(t)-\bx_i^\infty(t)|^2+ |\bw_i^c(t)-\bw_i^\infty(t)|^2 \Big)= 0.
\end{equation*}
\end{theorem}
\begin{proof}
For the desired estimate, we introduce a deviation functional:
\[
	\mathcal{D}(t) := \sum_{i=1}^N \Big( |\bx_i^c(t)-\bx_i^\infty(t)|^2+ |\bw_i^c(t)-\bw_i^\infty(t) |^2 \Big).
\]
Then, we can derive the following Gronwall's inequality for ${\mathcal D}$ (see Appendix \ref{App-C}):
\begin{equation} \label{D-7-2}
\begin{cases}
\displaystyle \frac{d}{dt}\mathcal{D}(t) \leq C\mathcal{D}(t) + \mathcal{O}(c^{-2}), \quad t > 0, \\
\displaystyle {\mathcal D}(0) = 0,
\end{cases}
\end{equation}
where $C$ is a positive constant independent of $t$. Then, we apply Gronwall's lemma to \eqref{D-7-2} to get 
 \[
 	\sup_{t \in [0,T]}\mathcal{D}(t) \leq \mathcal{O}(c^{-2})\int_0^T e^{C(T-s)}ds.
 \]
This yields 
\[ \lim_{c \nearrow \infty}\mathcal{D}(t)=0, \quad t \in (0, T).  \] 
\end{proof}


\section{The RCS model with a bonding force on manifolds} \label{sec:5}
\setcounter{equation}{0}
In this section, we recall basic terminologies and notation from differential geometry \cite{Ju}  that will be frequently used in what follows, and then  we discuss modeling spirit, collision avoidance and global well-posedness for the RCS model with a bonding force on manifolds. 
\subsection{Minimum materials for differential geometry}\label{sec:5.1}
Let $({\mathcal M}, g)$ be a connected, complete and  smooth $d$-dimensional Riemannian manifold without boundary with a metric tensor $g$. 
\subsubsection{Tangent vector, tangent space and tangent bundle} \label{sec:5.1.1} For each point $\bp \in \mathcal{M}$ and its small neighborhood $\mathcal{U}\subset \mathcal{M}$, let $({\mathcal U}, \bp)$ be a local chart so that the point $\bp$ can be assigned to local coordinates, say $\bx(\bp) = (x_1(\bp), \cdots, x_d(\bp)) \in \bbr^d$, and we define the tangent space of ${\mathcal M}$, denoted by $T_{\bp} {\mathcal M}$,  as the set of all $\bbr$-linear functional $X_{\bp}:C^\infty(\bp)\to\mathbb{R}$ satisfying the Leibniz rule:
\[
X_{\bp} (f_1f_2)=(X_{\bp} f_1)\cdot f_2(\bp)+f_1(\bp)\cdot(X_{\bp} f_2),\quad \forall~ f_1,f_2\in C^\infty(\bp),
\]
where $C^\infty(\bp)$ is the set of germs of $C^\infty$ functions at $\bp$.  Then,  the set $T_{\bp}{\mathcal M}$ can be regarded as a vector space for operations $(+,\cdot)$:
\[(X_{\bp} +Y_{\bp})f:=X_{\bp} f+Y_{\bp} f,\quad (\lambda \cdot X_{\bp})f:=\lambda \cdot(X_{\bp} f), \quad \forall~ \lambda \in \mathbb{R},~f\in C^\infty(\bp).\]
In fact, one can take the set $\left\{\frac{\partial}{\partial x_i}\big|_{\bp}: i=1,\cdots,d \right\}$ as a basis for $T_{\bp} {\mathcal M}$:
\[ \frac{\partial}{\partial x_i}\Big|_{\bp}:C^\infty(\bp)\to \mathbb{R},\quad f\mapsto \frac{d}{dt}\Big|_{t=0}f\left(x^{-1}\left(x(\bp)+t\mathbf{e}_i \right)\right),  \]
where $\mathbf{e}_i$ is the $i$-th element of the standard orthonormal basis in $\mathbb{R}^d$. On the other hand, we set the tangent bundle $T {\mathcal M}$ as 
\[ T{\mathcal M} := \{(\bp, \bv) \in {\mathcal M} \times T_{\bp} {\mathcal M}:~ \bp \in {\mathcal M},~~ \bv = X_{\bp} \in T_{\bp} {\mathcal M} \}. \]

\subsubsection{The Levi-Civita connection and parallel transport} \label{sec:5.1.2}  
Let ${\mathcal X}({\mathcal M})$ be the collection of $C^{\infty}$-vector fields on ${\mathcal M}$. Then, a smooth affine connection $\nabla$ on ${\mathcal M}$ is a $\mathbb{R}$-bilinear map: 
\[
\nabla:\mathcal{X}({\mathcal M})\times \mathcal{X}({\mathcal M})\to\mathcal{X}({\mathcal M}), \quad  (X,Y)\mapsto \nabla_X Y,
\] 
satisfying the following rules: for all $f\in C^\infty({\mathcal M}),~X,Y\in\mathcal{X}({\mathcal M})$ and $p\in {\mathcal M}$,
\begin{equation*}\label{rules2}
(\nabla_{fX}Y)_p=f(p)(\nabla_X Y)_p, \qquad (\nabla_X (fY))_p=f(p)(\nabla_X Y)_p+(X_pf)Y_p.
\end{equation*}
Let $\nabla$ be the Levi-Civita connection of $\mathcal{M}$ which is symmetric and compatible with the Riemannian metric tensor $g$, which is uniquely determined by $g$. If a map $\gamma:I\to {\mathcal M}$ is a smooth curve and $X(t)$ is a vector field along $\gamma$, then we call $\nabla_{\dot{\gamma}(t)}X$ as the covaraint derivative of $X$ along $\gamma$. By the compatibility of the Levi-Civita connection, if $X$ and $Y$ are vector fields along $\gamma$,
one has 
\begin{align*}\label{E-0}
\frac{d}{dt}g(X,Y)=g(X,\nabla_{\dot{\gamma}(t)}Y)+g(\nabla_{\dot{\gamma}(t)}X,Y).
\end{align*}
In particular, if $X$ and $\gamma:[0,T]\to {\mathcal M}$ satisfies
\begin{equation}\label{E-1}
\nabla_{\dot\gamma(t)}X=0,\quad t\in (0,T),
\end{equation}
we call the vector field $\nabla_{\dot\gamma(t)}X$ as a parallel vector field along $\gamma$. By the Cauchy-Lipschitz theory for ODE, one can find a unique solution $(x^1,\cdots,x^d)$ of \eqref{E-1} for given $\gamma$ and $X(0)$. That is, for each $(\bp,\bv)\in T{\mathcal M}$ and a curve $\gamma:[0,T]\to {\mathcal M}$ with $\gamma(0)=\bp$, there exists a unique parallel tangent vector field $X$ on $\gamma$ satisfying $X_{\bp}=\bv$, and this parallel transport mapping $P(\gamma)_0^t: \bv\mapsto P(\gamma)_0^t\bv:=X_{\gamma(t)}$ is linear for each $t\in [0,T]$.\\

\subsubsection{Geodesic and exponential map} \label{sec:5.1.3}  
If a tangent vector field of a curve $\gamma$ is parallel along $\gamma$, one has 
\begin{equation*}
\nabla_{\dot\gamma(t)}\dot\gamma=0,\quad t\in (0,T).
\end{equation*} 
In this case, we call the curve $\gamma$ by an (affine) geodesic on $\mathcal{M}$ for the Levi-Civita connection $\nabla$. Now, we briefly recall the Hopf-Rinow theorem which guarantees the well-definedness of the exponential map on the whole domain. Here, if a geodesic $\gamma$ of $({\mathcal M},g)$, corresponding to Levi-Civita connection with $\gamma(0)=\bp $ and $\dot\gamma(0)=\bv$ is well-defined at least for $0\leq t\leq 1$, we define an exponential map by
\[ \exp_{\bp}\bv:=\gamma(1). \]
To justify the inverse map of an {\it{exponential}} map, we note that it is well known that every exponential map is a local diffeomorphism. Thus, by the inverse function theorem, it defines the {\it{logarithm}} map by the local inverse $\exp^{-1}_{\bp}\bv$ of an exponential map. In other words, if an $\exp_{\bp}\bv$ is well-defined on a open set $V\subset T_{\bp}\mathcal{M}$, then the  corresponding logarithm mapping is 
\[\log_{\bp}{\bq}=\dot\gamma(0),\quad \bq\in \exp_{\bp}(V),\] where $\gamma$ is a geodesic curve satisfying $\gamma(0)=\bp$ and $\gamma(1)=\bq$. Then, the following proposition is the Hopf-Rinow theorem.
\begin{proposition}\label{P1.1}\emph{(Hopf-Rinow) \cite{Ju}}
	A connected and smooth Riemannian manifold $({\mathcal M},g)$ is (topologically) complete if and only if the exponential map $\exp_{\bp} \bv$ is well-defined for any $(\bp,\bv)\in T{\mathcal M}$, and this implies the existence of geodesics (possibly not unique) connecting any two points $\bx,\by$ on $({\mathcal M},g)$.
\end{proposition}
\noindent As a corollary, for a connected, smooth and complete Riemannian manifold $({\mathcal M},g)$, any two points $\bx,\by$ on $({\mathcal M},g)$ admit at least one minimizer $\gamma$ of length $\ell(\gamma)$ among the set
\[ \Big \{\omega(\alpha)=\bx,~\omega(\beta)=\by,~\omega~\mbox{is a piecewise smooth curve on}~{\mathcal M},~\alpha, \beta\in \mathbb{R} \Big \},\]
and this $\gamma$ is one of the geodesics joining $\bx$ and $\by$, which is therefore smooth. We call this $\gamma$ a length-minimizing geodesic joining $\bx$ and $\by$.

\subsubsection{Injectivity radius} \label{sec:5.1.4}  
In this part, we recall {\it {the injectivity radius}} of $\mathcal{M}$. Consider an open ball $B_{\mathcal{M}}(\bx,r)$ at $\bx\in \mathcal{M}$ is defined by 
\[B_{\mathcal{M}}(\bx,r):=\{\by |~d(\bx,\by)\leq r\},\] where $d$ is the length-minimizing geodesic distance between $\bx$ and $\by$. Here, an injectivity radius at $\bx$, $\text{inj}_{\bx}\mathcal{M}$ is the largest radius for which the exponential map at $\bx$ is a diffeomorphism. In addition, an injectivity radius of $\mathcal{M}$ is denoted by $\mathcal{M}_{radii}:=\inf_{\bx\in\mathcal{M}}\text{inj}_{\bx}\mathcal{M}$, which guarantees an existence of a unique length-minimizing geodesic distance bewteen two points in $\mathcal{M}$.

\subsection{Comments on modeling spirit} \label{sec:5.2}
In this subsection, we briefly discuss a modeling spirit so that one can see how the manifold model \eqref{A-8} can be lifted from the trivial manifold model \eqref{A-7}. For details, we refer to Appendix \ref{App-E}. \newline

For each $i \in [N],$ let $\bx_i:[0,\infty)\to {\mathcal  M}$ be a smooth curve representing the trajectory of  position for the $i$-th RCS particle and let $\bv_i$ be the tangent velocity vector of the $i$-th particle, respectively, and $T_{\bx_i} {\mathcal M}$ and $T{\mathcal M}$ are the tangent space of ${\mathcal M}$ at the foot point $\bx_i$ and tangent bundle of ${\mathcal M}$, respectively. \newline

Recall that our goal is to obtain a manifold counterpart for the RCS model \eqref{A-7} on $\bbr^d$. First, we require 
\begin{equation} \label{E-2-0}
  {\dot \bx}_i = \hat{v}(\bw_i) \in T_{\bx_i} {\mathcal M}, \quad \mbox{i.e.,} \quad  \bw_i \in T_{\bx_i} {\mathcal M} \quad \mbox{form} ~ \eqref{A-0}.
 \end{equation}
Second, we consider the momentum equation in RCS model \eqref{A-7} on the Euclidean space $\bbr^d$:
\begin{equation*}
\begin{cases}\label{E-2}
\displaystyle {\dot \bw}_i  =\frac{\kappa_0}{N}\sum_{j=1}^{N}\phi(|\bx_i-\bx_j|)\left(\hat{v}(\bw_j) - \hat{v}(\bw_i) \right) \\
\vspace{0.1cm}
\displaystyle \hspace{0.5cm} + \frac{1}{2N}\sum_{j \neq  i}
	\Big[ \kappa_1 \Big \langle  \hat{v}(\bw_j)- \hat{v}(\bw_i), \frac{\bx_j-\bx_i}{| \bx_i - \bx_j |} \Big \rangle  + \kappa_2 \Big(|\bx_i-\bx_j|-R^{\infty}_{ij} \Big)  \Big]  \frac{(\bx_j-\bx_i)}{| \bx_i - \bx_j|},
\end{cases}
\end{equation*}
The obvious manifold counterparts for ${\dot \bw}_i,~|\bx_j - \bx_i|$ and $\langle \cdot, \cdot \rangle$ will be 
\begin{equation} \label{E-2-1}
\nabla_{{\dot \bx}_i} \bw_i, \quad d(\bx_i, \bx_j) \quad \mbox{and} \quad  g_{\scriptsize{\bx_i}}(\cdot, \cdot), \quad \mbox{respectively.} 
\end{equation}
Thus, it remains to consider the following terms:
\[ 
\hat{v}(\bw_j) - \hat{v}(\bw_i), \quad \bx_j - \bx_i, \quad \mbox{for}~i, j \in [N]. 
\]
However, elements in $T\mathcal{M}$ are not compatible in general. In the sequel, these obstacles will be bypassed via the parallel transport and the logarithm map discussed in previous subsection. \newline

 \noindent $\bullet$~(Manifold counterpart of $ \hat{v}(\bw_j) - \hat{v}(\bw_i)$):~ Let $\nabla$ be the Levi-Civita connection compatible with $g$, and let $P_{ij} : T_{\bx_j} {\mathcal M} \to T_{\bx_i} {\mathcal M}$ be the parallel transport along the length minimizing geodesic from $\bx_j$ to $\bx_i$, which is well-defined only when the length-minimizing geodesic is unique. Then, it satisfies the following relations:~ for $\bu\in T_{\bx_i}\mathcal{M}$ and $\bu',\bv'\in T_{\bx_j}\mathcal{M}$,
\begin{align}
\begin{aligned} \label{E-3}
& P_{ij}P_{ji}=\mathrm{Id},~~\|P_{ij}\bu'\|_{\bx_i}=\|\bu'\|_{\bx_j}, \\
&g_{\scriptsize{\bx_i}}(P_{ij}\bu',P_{ij}\bv') =g_{\scriptsize{\bx_j}}(\bu',\bv'),~~\|\bu\|_{\bx_i}: = \sqrt{g_{\scriptsize{\bx_i}} (\bu,\bu)}. 
\end{aligned}
\end{align}
Thus, it is reasonable to use the following replacement:
\begin{equation} \label{E-3-1}
 \hat{v}(\bw_j) - \hat{v}(\bw_i) \quad \Longrightarrow \quad P_{ij}  \hat{v}(\bw_j) -  \hat{v}(\bw_i) \in T_{\bx_i} {\mathcal M}. 
\end{equation}

\vspace{0.2cm}

\noindent $\bullet$~(Manifold counterpart of $\bx_j-\bx_i$):~Now, we consider how to modify $\bx_j - \bx_i$ so that the resulting relation lies in ${\mathcal M}$.  Let $\gamma$ be the length-minimizing geodesic from $\bx_i$ to $\bx_j$, and we set $d_{ij} := d(\bx_i,\bx_j)$, and $\log_{\bx_i}\bx_j$ be the logarithm mapping along with $\gamma$.  Then, it is reasonable to use the following replacement:
\begin{equation} \label{E-3-2}
\bx_j - \bx_i  \quad \Longrightarrow \quad \log_{\bx_i} {\bx_j}.
\end{equation}
Now, we gather all the manifold counterparts \eqref{E-2-1}, \eqref{E-3-1} and \eqref{E-3-2} to see the manifold counterpart of the momentum equation in  RCS model with a bonding force:
\begin{equation}
\begin{cases} \label{E-7}
\displaystyle \nabla_{{\dot \bx}_i} \bw_i=\frac{\kappa_0}{N}\sum_{j=1}^{N}\phi(d(\bx_i, \bx_j))\left(P_{ij} \hat{v}(\bw_j) - \hat{v}(\bw_i) \right)\\
\displaystyle + \frac{1}{2N}\sum_{j \neq i} \Bigg[  \kappa_1 g_{\scriptsize \bx_i} \Big( P_{ij} \hat{v}(\bw_j) -\hat{v}(\bw_i), \frac{\log_{\bx_i}\bx_j}{d(\bx_i, \bx_j)} \Big)  + \kappa_2(d(\bx_i, \bx_j) -R^\infty_{ij}) \Bigg] \frac{\log_{\bx_i}\bx_j}{d(\bx_i, \bx_j)},\\
\displaystyle (\bx_i(0),\bw_i(0))=(\bx_i^0,\bw_i^0)\in T\mathcal{M}.
\end{cases}
\end{equation}
By the construction of \eqref{E-7}, it is easy to see that 
\[ \bw_i^0 \in T_{\bx_i^0} {\mathcal M} \quad \Longrightarrow \quad \bw_i(t) \in T_{\bx_i(t)} {\mathcal M} \]
which guarantees \eqref{E-2-0}.  The systematic derivation of bonding force appearing in the R.H.S. of $\eqref{E-7}$ will be provided in Appendix \ref{App-E}. The communication weight function $\phi : \bbr_+ \cup \{0\} \to \bbr_+ \cup \{0\}$ satisfies \eqref{PA-2}.
In particular, if $d(\bx_i, \bx_j)<\mathcal{M}_{radii}$ and $\max_{i}\|\hat{v}(\bw_i)\|$ is uniformly bounded by some positive value strictly less than $c$ with the collision avoidance in \eqref{E-7}, then the velocity coupling and bonding force terms are bounded locally Lipschitz continuous. Thus, we can verify the global well-posedness of \eqref{A-8} by the standard Cauchy-Lipschitz theory for ODE on manifolds. A sufficient framework for the global well-posedness will be treated in Section \ref{sec:5.3}.

\subsection{A global well-posedness} \label{sec:5.3}
In this subsection, we discuss a global well-posedness of \eqref{A-8}. Recall that every manifolds do not need to admit a unique length-minimizing geodesic in general. Moreover, to be consistent with special relativity, particle's speed should be strictly smaller than the speed of light, and due to the presence of $d(\bx_i, \bx_j)$ in the denominator of \eqref{A-8}, we have to make sure collision avoidance in a finite time interval. These subtle issues will be cleared out in what follows.  First, we introduce a manifold counterpart for the relativistic mechanical energies as follows.
\begin{definition}\label{E5.1}
	\emph{(Manifold counterpart of relativistic energy)} Let $\{(\bx_i,\bw_i)\}$ be a solution to the Cauchy problem \eqref{A-8}. Then, the relativistic kinetic, potential and total energies 
	${\mathcal E}^c_k$, ${\mathcal E}^c_p$ and ${\mathcal E}^c$ are defined as follows:
	\begin{align*}
	\begin{aligned} \label{E-7-0}
	& \Gamma_i :=  \frac{1}{\sqrt{1-\frac{\|\bv_i \|_{\scriptsize{\bx}_i}^2}{c^2}}}, \qquad \mathcal{E}_{{\mathcal M}, k}^c:=\sum_{i=1}^N  \Big[ c^2(\Gamma_i-1)+\big( \Gamma_i^2-\log{\Gamma_i} \big) \Big], \\
	& \mathcal{E}^c_{{\mathcal M},p} :=\frac{\kappa_2}{8N}\sum_{i,j=1}^{N}(d(\bx_i, \bx_j) -R^\infty_{ij})^2, \qquad  \mathcal{E}_{{\mathcal M}}^c:=\mathcal{E}^c_{{\mathcal M},k}+\mathcal{E}^c_{{\mathcal M},p},
	\end{aligned}
	\end{align*}
	where $d(\bx_i, \bx_j)$ is a length minimizing distance between $\bx_i$ and $\bx_j$. 
\end{definition}
\begin{remark}
Since $\Gamma_i$ and $d(\bx_i, \bx_j)$ depend on the metric tensor, energies $ \mathcal{E}_{{\mathcal M}, k}^c$ and $\mathcal{E}^c_{{\mathcal M},p}$ also depend on the metric tensor of  ${\mathcal M}$ and $c$. 
\end{remark}
Before we perform an energy estimate for \eqref{E-7}, we study some elementary estimates as follows. 
\begin{lemma}\label{L5.1}
Let $\bx_i(\cdot)$ and $\bx_j(\cdot)$ be smooth curves on $\mathcal{M}$ and $d(\cdot,\cdot): \mathcal{M}\times \mathcal{M}\rightarrow \mathbb{R}_{+} \cup \{0\}$ be a length-minimizing geodesic distance on $(\mathcal{M},g)$ and $(\bx_i,\bv_{i}:=\dot{\bx}_{i}),(\bx_j,\bv_{j}:=\dot{\bx}_{j}) \in T\mathcal{M}$. Furthermore, we assume that 
\[ d(\bx_i,\bx_j)<\mathcal{M}_{radii}, \]
and $P_{ij}$ is the parallel transport from $T_{\bx_j}\mathcal{M}$ to $T_{\bx_i}\mathcal{M}$ connecting $\bx_j$ to $\bx_i$. Then, one has
\begin{enumerate}
	\item $\displaystyle P_{ij}(\log_{\bx_j}\bx_i)=-\log_{\bx_i}\bx_j$,\quad $\displaystyle P_{ji}(\log_{\bx_i}\bx_j)=-\log_{\bx_j}\bx_i$,
	\item $\displaystyle \frac{d}{dt} d^2(\bx_i,\bx_j) =2d(\bx_i,\bx_j)\dot{d}(\bx_i,\bx_j)=2g_{\scriptsize \bx_i} (P_{ij}\bv_j-\bv_i,\log_{\bx_i}\bx_j), \quad \mbox{a.e.~$t > 0$}.$
\end{enumerate}
\end{lemma}
\begin{proof} We leave its proof in Appendix \ref{App-D}.
\end{proof}
Now, we are ready to perform an energy estimate for system \eqref{A-8}. First, we take an inner product $\bv_i$ with $\eqref{A-8}_{2}$ to find 
\begin{equation} \label{E-7-1}
g_{\scriptsize \bx_i}(\nabla_{{\dot \bx}_i} \bw_i, \bv_i)  =  g_{\scriptsize \bx_i} (\mbox{the R.H.S. of $\eqref{A-8}_2$}, \bv_i).
\end{equation}
\vspace{0.1cm}

\noindent $\bullet$~(Estimate on the L.H.S. of \eqref{E-7-1}):  We use \eqref{A-0}  to obtain
\begin{align}
\begin{aligned} \label{E-7-2}
g_{\scriptsize \bx_i}(\nabla_{{\dot \bx}_i} \bw_i, \bv_i) &= g_{\scriptsize \bx_i} (\nabla_{\bv_i} F \bv_i, \bv_i) = g_{\scriptsize \bx_i} (F \nabla_{\bv_i} \bv_i + \frac{dF}{dt} \bv_i, \bv_i) \\
&= F g_{\scriptsize \bx_i} (\nabla_{\bv_i} \bv_i,  \bv_i)+ \frac{dF}{dt} g_{\scriptsize \bx_i} (\bv_i, \bv_i) \\
&= g_{\scriptsize \bx_i} (\nabla_{{\bv}_i}\bv_i,\bv_i) \bigg(\Gamma_i\bigg(1+\frac{\Gamma_i}{c^2}\bigg)\bigg)+\|\bv_i\|_{\bx_i}^2\frac{d}{dt}\bigg(\Gamma_i\bigg(1+\frac{\Gamma_i}{c^2}\bigg)\bigg).
\end{aligned}
\end{align}
\vspace{0.2cm}
\noindent $\bullet$~(Estimate on the R.H.S .of \eqref{E-7-1}):  By direct calculation, one has 
\begin{align}
\begin{aligned} \label{E-8}
&g_{\scriptsize \bx_i} (\mbox{the R.H.S. of $\eqref{E-7}_2$}, \bv_i) \\
& \hspace{0.5cm} = \frac{\kappa_0}{N}\sum_{j=1}^{N} \phi(d(\bx_i,\bx_j))g_{\scriptsize \bx_i} ( P_{ij}\bv_j- \bv_i,\bv_i) \\
&\hspace{1cm}+\displaystyle\frac{\kappa_1}{N}\sum_{j \neq i}\frac{1}{2d^2_{ij}}g_{\scriptsize \bx_i} (P_{ij}\bv_j-\bv_i,\log_{\bx_i}\bx_j) g_{\scriptsize \bx_i} (\log_{\bx_i}\bx_j,\bv_i) \\
&\hspace{1cm}+\displaystyle\frac{\kappa_2}{N}\sum_{j \neq i}\frac{d_{ij}-R^\infty_{ij}}{2d_{ij}}g_{\scriptsize \bx_i} (\log_{\bx_i}\bx_j,\bv_i),
\end{aligned}
\end{align}
where we denoted $d_{ij}:=d(\bx_i,\bx_j)$ for the simplicity. We add \eqref{E-7-1} over all $i \in [N]$ by combining \eqref{E-7-2} and \eqref{E-8} to obtain
\begin{align} 
\begin{aligned} \label{E-9}
&\sum_{i=1}^{N} \Big[ g_{{\scriptsize \bx_i}}(\nabla_{{\bv}_i}\bv_i,\bv_i) \bigg(\Gamma_i\bigg(1+\frac{\Gamma_i}{c^2}\bigg)\bigg)+\|\bv_i\|_{\bx_i}^2\frac{d}{dt}\bigg(\Gamma_i\bigg(1+\frac{\Gamma_i}{c^2}\bigg)\bigg) \Big] \\
& \hspace{1cm} = \frac{\kappa_0}{N}\sum_{i,j=1}^{N} \phi(d(\bx_i,\bx_j))g_{\scriptsize \bx_i} ( P_{ij}\bv_j- \bv_i,\bv_i) \\
&\hspace{1.5cm}+\displaystyle\frac{\kappa_1}{N}\sum_{\substack{j,i=1 \\ i \neq j }}^{N} \frac{1}{2d^2_{ij}}g_{\scriptsize \bx_i} (P_{ij}\bv_j-\bv_i,\log_{\bx_i}\bx_j) g_{\scriptsize \bx_i} (\log_{\bx_i}\bx_j,\bv_i)\\
&\hspace{1.5cm}+\displaystyle\frac{\kappa_2}{N}\sum_{\substack{j,i=1 \\ i \neq j }}^{N} \frac{d_{ij}-R^\infty_{ij}}{2d_{ij}}g_{\scriptsize \bx_i} (\log_{\bx_i}\bx_j,\bv_i) \\
& \hspace{1cm} =: {\mathcal I}_{11} +  {\mathcal I}_{12} +  {\mathcal I}_{13}.
\end{aligned}
\end{align}
In the following lemma, we provide estimates for ${\mathcal I}_{1i}$.
\begin{lemma}\label{L5.2}
For $\tau\in(0,\infty]$, let $\{(\bx_i,\bw_i)\}$ be a solution to \eqref{A-8} on $t\in[0,\tau)$. Then, one has the following estimates:
\begin{align*}
\begin{aligned}
& (i)~ {\mathcal I}_{11} = -\frac{\kappa_0}{2N}\sum_{\substack{j,i=1 \\ i \neq j }}^{N} \phi(d(\bx_i,\bx_j)) \| P_{ij}\bv_j- \bv_i\|_{\bx_i}^2,  \\
& (ii)~  {\mathcal I}_{12} = -\frac{\kappa_1}{4N}\sum_{\substack{j,i=1 \\ i \neq j }}^{N}\frac{|g_{\scriptsize \bx_i} (P_{ij}\bv_j-\bv_i,\log_{\bx_i}\bx_j)|^2}{d^2_{ij}}, \\
& (iii)~ {\mathcal I}_{13} =  -\frac{\kappa_2}{4N}\sum_{\substack{j,i=1 \\ i \neq j }}^{N}\frac{d_{ij}-R^\infty_{ij}}{d_{ij}}g_{\scriptsize \bx_i} (P_{ij}\bv_j-\bv_i,\log_{\bx_i}\bx_j).
\end{aligned}
\end{align*}
\end{lemma}
\begin{proof} \noindent $\diamond$~(Estimate of $ {\mathcal I}_{11}$):  We use the index switching $i \leftrightarrow j$ and relations in \eqref{E-3}  to get 
\begin{align*}
\begin{aligned} \label{E-10}
 {\mathcal I}_{11} &=  \frac{\kappa_0}{N}\sum_{i,j=1}^{N} \phi(d(\bx_i,\bx_j))g_{\scriptsize \bx_i} ( P_{ij}\bv_j- \bv_i,\bv_i) \\
 & =  \frac{\kappa_0}{N}\sum_{i,j=1}^{N} \phi(d(\bx_i,\bx_j)) g_{\scriptsize \bx_j} ( P_{ji}\bv_i- \bv_j,\bv_j) \quad \mbox{by $(i,j) \leftrightarrow (j,i)$} \\
 & =   \frac{\kappa_0}{N}\sum_{i,j=1}^{N} \phi(d(\bx_i,\bx_j)) g_{\scriptsize \bx_i} ( P_{ij} (P_{ji}\bv_i- \bv_j), P_{ij} \bv_j)  \quad  \mbox{by $\eqref{E-3}_3$} \\
 & =  -\frac{\kappa_0}{N}\sum_{i,j=1}^{N} \phi(d(\bx_i,\bx_j)) g_{\scriptsize \bx_i} ( P_{ij} \bv_j - \bv_i, P_{ij} \bv_j) \quad  \mbox{by $P_{ij} P_{ji} = Id$}  \\
 & = -\frac{\kappa_0}{2N}\sum_{\substack{j,i=1 \\ i \neq j }}^{N} \phi(d(\bx_i,\bx_j)) \| P_{ij}\bv_j- \bv_i\|_{\scriptsize \bx_i}^2 \quad  \mbox{by $\eqref{E-3}_4$}.
\end{aligned}
\end{align*}
\vspace{0.2cm}

\noindent $\diamond$~(Estimate of $ {\mathcal I}_{12}$): By similar index switching and Lemma 5.1 (i), one has 
\begin{align*}
\begin{aligned} 
 {\mathcal I}_{12} &= \frac{\kappa_1}{2N}\sum_{\substack{j,i=1 \\ i \neq j }}^{N} \frac{1}{d^2_{ij}}g_{\scriptsize \bx_i} (P_{ij}\bv_j-\bv_i,\log_{\bx_i}\bx_j) \cdot g_{\scriptsize \bx_i} (\bv_i, \log_{\bx_i}\bx_j)  \\
 & = \frac{\kappa_1}{2N}\sum_{\substack{j,i=1 \\ i \neq j }}^{N} \frac{1}{d^2_{ji}}g_{\scriptsize \bx_j} (P_{ji}\bv_i-\bv_j,\log_{\bx_j}\bx_i) \cdot g_{\scriptsize \bx_j} ( \log_{\bx_j}\bx_i, \bv_j)  \\
 & = \frac{\kappa_1}{2N}\sum_{\substack{j,i=1 \\ i \neq j }}^{N} \frac{1}{d^2_{ji}}g_{\scriptsize \bx_i} (P_{ij} (P_{ji}\bv_i-\bv_j), P_{ij} \log_{\bx_j}\bx_i) \cdot g_{\scriptsize \bx_i} (P_{ij} \log_{\bx_j}\bx_i, P_{ij} \bv_j) \\
& =  \frac{\kappa_1}{2N}\sum_{\substack{j,i=1 \\ i \neq j }}^{N} \frac{1}{d^2_{ij}}g_{\scriptsize \bx_i} (\bv_i- P_{ij} \bv_j, -\log_{\bx_i}\bx_j) \cdot g_{\scriptsize \bx_i} (-\log_{\bx_i}\bx_j, P_{ij} \bv_j) \\
& =   -\frac{\kappa_1}{2N}\sum_{\substack{j,i=1 \\ i \neq j }}^{N} \frac{1}{d^2_{ij}}g_{\scriptsize \bx_i} (P_{ij} \bv_j -\bv_i, \log_{\bx_i}\bx_j) \cdot g_{\scriptsize \bx_i} (P_{ij} \bv_j, \log_{\bx_i}\bx_j) \\
& =  -\frac{\kappa_1}{4N}\sum_{\substack{j,i=1 \\ i \neq j }}^{N}\frac{|g_{\scriptsize \bx_i} (P_{ij}\bv_j-\bv_i,\log_{\bx_i}\bx_j)|^2}{d^2_{ij}}.
 \end{aligned}
 \end{align*}

\vspace{0.2cm}

\noindent $\diamond$~(Estimate of $ {\mathcal I}_{13}$): Similar to ${\mathcal I}_{12}$, we have
\begin{equation*} \label{E-12}
 {\mathcal I}_{13}  = -\frac{\kappa_2}{4N}\sum_{\substack{j,i=1 \\ i \neq j }}^{N}\frac{d_{ij}-R^\infty_{ij}}{d_{ij}}g_{\scriptsize \bx_i} (P_{ij}\bv_j-\bv_i,\log_{\bx_i}\bx_j).
\end{equation*}
\end{proof}
Next, we provide an energy estimate for \eqref{E-7} using Lemma \ref{L5.2}.
\begin{proposition}\label{P5.2}
For $\tau\in(0,\infty]$, let $\{(\bx_i,\bw_i)\}$ be a solution to \eqref{A-8} in the time interval $[0,\tau)$. Then, the total energy $\mathcal{E}_{\mathcal M}^c$ satisfies
\[\mathcal{E}_{{\mathcal M}}^c(t)+\int_{0}^{t} {\mathcal P}_{{\mathcal M}}^c(s)ds=\mathcal{E}_{{\mathcal M}}^c(0)\quad\text{for}\quad t\in[0,\tau),\] 
where the total energy production rate is defined as follows:
\[
{\mathcal P}^c_{{\mathcal M}}(t) :=\frac{\kappa_0}{2N}\sum_{i,j=1}^{N} \phi(d(\bx_i,\bx_j)) \| P_{ij} \hat{v}(\bw_j) - \hat{v}(\bw_i) \|_{\scriptsize \bx_i}^2 + \frac{\kappa_1}{4N}\sum_{\substack{j,i=1 \\ i \neq j }}^{N}\frac{|g_{\scriptsize \bx_i}(P_{ij} \hat{v}(\bw_j) - \hat{v}(\bw_i),\log_{\bx_i}\bx_j)|^2}{d^2_{ij}}.
\]
\end{proposition}
\begin{proof} 
In \eqref{E-9}, we use Lemma \ref{L5.2} to find 
\begin{align} 
\begin{aligned} \label{E-13}
&\sum_{i=1}^{N} \Big[ g_{{\scriptsize \bx_i}}(\nabla_{{\bv}_i}\bv_i,\bv_i) \bigg(\Gamma_i\bigg(1+\frac{\Gamma_i}{c^2}\bigg)\bigg)+\|\bv_i\|_{\bx_i}^2\frac{d}{dt}\bigg(\Gamma_i\bigg(1+\frac{\Gamma_i}{c^2}\bigg)\bigg) \Big] \\
& \hspace{0.2cm} =-\frac{\kappa_0}{2N}\sum_{\substack{j,i=1 \\ i \neq j }}^{N} \phi(d(\bx_i,\bx_j)) \| P_{ij}\bv_j- \bv_i\|_{\scriptsize \bx_i}^2 -\frac{\kappa_1}{4N}\sum_{\substack{j,i=1 \\ i \neq j }}^{N}\frac{|g_{\scriptsize \bx_i}(P_{ij}\bv_j-\bv_i,\log_{\bx_i}\bx_j)|^2}{d^2_{ij}}\\
&\hspace{0.6cm}-\frac{\kappa_2}{4N}\sum_{\substack{j,i=1 \\ i \neq j }}^{N}\frac{d_{ij}- R^\infty_{ij}}{d_{ij}}g_{\scriptsize \bx_i} (P_{ij}\bv_j-\bv_i,\log_{\bx_i}\bx_j).
\end{aligned}
\end{align}
Next, we claim: 
\begin{align*}
\begin{aligned} \label{E-14}
& (i)~\frac{d}{dt}\mathcal{E}^c_{{\mathcal M},k}(t) = \mbox{L.H.S. of \eqref{E-13}}, \\
& (ii)~\frac{d}{dt} \mathcal{E}^c_{{\mathcal M},p}(t)=\frac{\kappa_2}{4N}\sum_{\substack{j,i=1 \\ i \neq j }}^{N}\frac{d_{ij}- R^\infty_{ij}}{d_{ij}}g_{\scriptsize \bx_i}(P_{ij}\bv_j-\bv_i,\log_{\bx_i}\bx_j).
\end{aligned}
\end{align*}
\noindent $\bullet$~(Derivation of (i)):~It follows from \eqref{A-0}  and \eqref{E-3} that 
\begin{equation} \label{E-15}
 g_{\scriptsize \bx_i} (\bv_i,\bv_i)= \|\bv_i\|^2_{\scriptsize \bx_i} = c^2 - \frac{c^2}{\Gamma_i^2}.
 \end{equation}
We differentiate \eqref{E-15} with respect to $t$ to obtain 
\begin{equation} \label{E-16}
\frac{d}{dt}  g_{\scriptsize{\bx_i}} (\bv_i,\bv_i)=  \frac{d}{dt} \Big( c^2 - \frac{c^2}{\Gamma_i^2} \Big)  \quad \Longleftrightarrow \quad  g_{\scriptsize \bx_i}(\nabla_{\bv_i} \bv_i, \bv_i) = \frac{c^2\dot{\Gamma}_i}{\Gamma_i^3}.
\end{equation}
Now, we use \eqref{E-16} to get
\begin{align}\label{E-17}
\begin{aligned}
&\sum_{i=1}^{N}g_{\scriptsize \bx_i}(\nabla_{{\bv}_i}\bv_i,\bv_i) \bigg(\Gamma_i\bigg(1+\frac{\Gamma_i}{c^2}\bigg)\bigg)+\|\bv_i\|_{\footnotesize \bx_i}^2\frac{d}{dt}\bigg(\Gamma_i\bigg(1+\frac{\Gamma_i}{c^2}\bigg)\bigg) \\
&\hspace{0.5cm} =\sum_{i=1}^{N}\bigg[ c^2+\Gamma_i\bigg(2-\frac{1}{\Gamma_i^2}\bigg)\bigg ] \frac{d\Gamma_i}{dt} \quad \mbox{by \eqref{E-15} and \eqref{E-16}} \\
&\hspace{0.5cm}=\sum_{i=1}^{N}\frac{d}{dt}\bigg[  c^2\bigg(\Gamma_i-1\bigg)+\bigg(\Gamma_i^2-\log{\Gamma_i}\bigg)\bigg ] =\frac{d}{dt}\mathcal{E}^c_{{\mathcal M},k}(t).
\end{aligned}
\end{align}
\vspace{0.2cm}
\noindent $\bullet$~(Derivation of (ii)): We use Lemma \ref{L5.1} (ii) to obtain the desired estimate:
\[\frac{d\mathcal{E}^c_{{\mathcal M},p}(t)}{dt} = \frac{\kappa_2}{8N}\sum_{i,j=1}^{N} \frac{d}{dt} (d_{ij} -R^\infty_{ij})^2 = \frac{\kappa_2}{4N}\sum_{\substack{j,i=1 \\ i \neq j }}^{N}\frac{d_{ij}-R^\infty_{ij}}{d_{ij}}g_{\scriptsize \bx_i}(P_{ij}\bv_j-\bv_i,\log_{\bx_i}\bx_j).\] 
Finally, in \eqref{E-13}, we use \eqref{E-17} to derive energy estimate:
\begin{align*} 
\frac{d}{dt}\mathcal{E}^c_{\mathcal M}(t) &=\frac{d}{dt} \mathcal{E}^c_{{\mathcal M}, k}(t) +\frac{d}{dt} \mathcal{E}^c_{{\mathcal M}, p}(t)\\
&=-\frac{\kappa_0}{2N}\sum_{i,j=1}^{N} \phi(d(\bx_i,\bx_j)) \|P_{ij} \bv_j- \bv_i\|_{\scriptsize \bx_i}^2-\displaystyle\frac{\kappa_1}{4N}\sum_{\substack{j,i=1 \\ i \neq j }}^{N}\frac{g_{\scriptsize \bx_i}(P_{ij}\bv_j-\bv_i,\log_{\bx_i}\bx_j)^2}{d^2_{ij}}\\
&=-{\mathcal P}_{{\mathcal M}}^c(t).
\end{align*}
\end{proof}
\begin{remark} In Proposition \ref{P5.2}, it follows from the second assertion of Lemma \ref{L5.1} that the total energy production can be rewritten as 
	\[ {\mathcal P}^c_{\mathcal M}(t):=\frac{\kappa_0}{2N}\sum_{i,j=1}^{N} \phi(d(\bx_i,\bx_j)) \|P_{ij} \bv_j- \bv_i\|_{\scriptsize \bx_i}^2+\displaystyle\frac{\kappa_1}{4N}\sum_{\substack{j,i=1 \\ i \neq j }}^{N}\left(\frac{d}{dt} d(\bx_i, \bx_j) \right)^2.\]
\end{remark}

\vspace{0.5cm}

In the following lemma, we study lower and upper bound estimates for the relative distances and estimate for particle's speed. 
\begin{lemma}\label{L5.3}
For $\tau\in(0,\infty]$, let $\{(\bx_i,\bw_i)\}$ be a local-in-time solution to \eqref{E-7} on $t\in[0,\tau)$. Then for any $i,j=[N]$, we have
\[ \underline{r}  \leq d(\bx_i(t), \bx_j(t)) \leq \overline{r}, \qquad \sup_{t \in [0,\tau)}\max_{i}\|\bv_i(t)\|_{\bx_i} <c, \]
where $\underline{r}$ and $\overline{r}$ are positive constants appearing in \eqref{C-0} where $\mathcal{E}^c(0)$ is replaced by ${\mathcal E}^c_{{\mathcal M}
}(0)$. 
\end{lemma}
\begin{proof}
\noindent To prove the inequalities in the first relation, recall that 
\[ \mathcal{E}^c_{{\mathcal M}, p}:=\frac{\kappa_2}{8N}\sum_{i,j=1}^{N}(d_{ij} -R^\infty_{ij})^2. \]
It follows from Proposition \ref{P5.2} that
	\[N+\mathcal{E}^c_{{\mathcal M},p}(t)\leq\mathcal{E}^c_{{\mathcal M},k}(t)+\mathcal{E}^c_{{\mathcal M}, p}(t)=\mathcal{E}^c_{{\mathcal M},k}(t)+\frac{\kappa_2}{8N}\sum_{i,j=1}^{N}(d_{ij}(t)- R^\infty_{ij})^2 = {\mathcal E}_{\mathcal M}^c(t) \leq \mathcal{E}^c_{{\mathcal M}}(0),\] 
for $t \in [0, \tau)$. This implies
	\[\frac{\kappa_2}{4N}|d_{ij}(t)-R^\infty_{ij}|^2\le\frac{\kappa_2}{8N}\sum_{i,j=1}^N|d_{ij}(t)-R^\infty_{ij}|^2=\mathcal{E}_{{\mathcal M}}^c(t)-\mathcal{E}_{{\mathcal M}, k}^c(t)\le\mathcal{E}^c_{{\mathcal M}}(0)-N.\]
	
\noindent To verify the second inequality, we first use Proposition \ref{P5.2} to see
	\[\mathcal{E}^c_{{\mathcal M},k}(t)\leq {\mathcal E}^c_{{\mathcal M}}(t) \leq \mathcal{E}^c_{{\mathcal M}}(0).\] 
	If we regard $\mathcal{E}_i(t)$ as a functional with respect to $\|\bv_i(t)\|_{\bx_i}$, there exists a strictly increasing function $f:[0,c)\rightarrow [1,\infty)$ given by
	\[x\mapsto c^2\left(\frac{c}{\sqrt{c^2-x^2}}-1\right)+\left(\frac{c^2}{c^2-x^2}-\log\left(\frac{c}{\sqrt{c^2-x^2}}\right)\right)\] such that 
	$\mathcal{E}_{{\mathcal M},k}^c(t)=f(\|\bv_i(t)\|_{\bx_i})\leq \mathcal{E}_{{\mathcal M}}^c(0).$ Therefore, we have the desired result:
	\[\|\bv_i(t)\|_{\bx_i}\leq f^{-1}(\mathcal{E}_{{\mathcal M}}^c(0))<c. \]
\end{proof}
Finally, for the well-posedness of \eqref{E-7}, we verify that the parallel transport operator along a unique length-minimizing geodesic is guaranteed in a local time interval by using the definition of injectivity radius.

\begin{lemma}\label{L5.4}
For $\tau\in(0,\infty]$, let $\{(\bx_i,\bw_i)\}$ be a local-in-time solution to \eqref{A-8} on $t\in[0,\tau)$. Suppose that the injectivity radius of the manifold satisfies
	\begin{align}\label{E-18}
	\overline{r} < \mathcal{M}_{radii}.
	\end{align}
	  Then, the parallel transport operator $P_{ij}$ along a unique length-minimizing geodesic is well-defined for all $t\in [0,\tau)$ and $i,j \in [N]$.
\end{lemma}
\begin{proof} It follows from Lemma \ref{L5.3} (i) and \eqref{E-18} that  
\[d_{ij}(t)<\mathcal{M}_{radii},\quad t\in [0,\tau).\]
Since ${\mathcal M}$ is complete, the exponential map $\exp_{\bx}:T_{\bx}\mathcal{M}\to \mathcal{M}$ is a diffeomorphism for all $\bx \in \mathcal{M}$. Thus, there is a unique length-minimizing geodesic between $\bx$ and $\by$ for all $(\bx,\by)\in \mathcal{M}\times \mathcal{M}$. Therefore, it follows from the definition of injectivity radius that we have the desired result.
\end{proof}
Note that, parallel to Lemma \ref{L3.1}, strictly positive lower bound of distance between particles for \eqref{E-8} is obtained if 
\[\displaystyle \underline{r} = \min_{i \neq j}R^\infty_{ij}-\sqrt{\frac{4N(\mathcal{E}_{\mathcal M}^c(0)-N)}{\kappa_2}}>0\] from Lemma $\ref{L5.3}$, and accordingly, collision does not occur. In addition, the well-definedness of length-minimizing geodesic can be achieved under the suitable initial parameters, which can be summarized in the following corollary.

\begin{corollary}\label{C5.1}
For $\tau\in(0,\infty]$, let $\{(\bx_i,\bw_i)\}$ be a solution to \eqref{A-8} in the time interval $[0,\tau)$. 
\begin{enumerate}
\item Suppose that initial data satisfiy
	\begin{align*}\label{E-19}
	\mathcal{E}_{\mathcal M}^c(0)<N+\frac{\kappa_2(\min_{i \neq j}R^\infty_{ij})^2}{4N},
	\end{align*}
	then one has
	\[
		\underline{r} > 0.
	\]
	In particular, collision does not occur.

\vspace{0.2cm}

\item Suppose that system parameters and initial data satisfy
	\begin{align*}
	\max_{i \neq j}R^\infty_{ij} < \mathcal{M}_{raddi},
	\quad \mathcal{E}_{\mathcal M}^c(0)<N+\frac{\kappa_2(\mathcal{M}_{raddi}- \max_{i \neq j}R^\infty_{ij})^2}{4N},
	\end{align*}
	then we have
	\[
		\overline{r} < \mathcal{M}_{raddi}.
	\]
	In particular, the unique length-minimizing geodesic between $\bx_i$ and $\bx_j$ is well-defined for any $i,j \in [N]$ and $t \in (0,\tau)$.
\end{enumerate}
\end{corollary}
\begin{proof} 
These are direct consequences of Lemma \ref{L5.3} and Lemma \ref{L5.4}.
\end{proof}

Finally, we are ready to discuss a global well-posedness of \eqref{E-7}.
\begin{theorem}\label{T5.1}
Suppose that system parameters and initial data satisfy
\begin{equation} \label{E-19}
	\max_{i \neq j}R^\infty_{ij} < \mathcal{M}_{raddi}, \quad
	\mathcal{E}_{{\mathcal M}}^c(0)<N+\frac{\kappa_2}{4N}
	\min\left\{ \min_{i \neq j}R^\infty_{ij} ~ , ~ \mathcal{M}_{raddi} - \max_{i \neq j}R^\infty_{ij}\right\}^2.
\end{equation}
    Then, the Cauchy problem \eqref{E-8} is globally well-posed. 
\end{theorem}
\begin{proof}
 We combine Corollary \ref{C5.1}, Lemma \ref{L5.3} and Lemma \ref{L5.4} to derive a collision avoidance in a finite-time interval. Then, as long as we can guarantee the nonexistence of finite-time collisions, the R.H.S. of \eqref{E-7} is locally Lipschitz continuous. Thus, we can still use the standard Cauchy-Lipschitz theory to derive a global well-posedness of \eqref{E-7}. 
\end{proof}
\begin{remark}
Unlike to the RCS model on manifolds without bonding force, we do not impose any a priori condition on solutions (see \eqref{E-19}). This is a positive transparent effect of bonding force. 
\end{remark}

\section{Conclusion} \label{sec:6}
In this paper, we have provided two nontrivial extensions of the Cucker-Smale model with a bonding force on the Euclidean space to the relativistic and manifold settings. In authors' recent work \cite{A-B-H-Y}, the authors observed the possibility of spatial pattern formations using the Cucker-Smale model by adding a bonding force with target relative spatial separations a priori. In this way, we can visualize the formation of the prescribed patterns using the Cucker-Smale model dynamically. In this work, we provided two extensions by adding relativistic and manifold effects. Whether the geometric structures of underlying manifold can hinder emergent dynamics or enforce emergent dynamics would be an interesting question. For this, we follow a systematic approach in \cite{A-B-H-Y} to address aforementioned physical and geometric effects. For two proposed extensions of the Cucker-Smale model, we provide several sufficient frameworks leading to the collision avoidance and uniform boundedness of relative distances. As a direct corollary of collision avoidance, we can derive a global well-posedness of the proposed model without any a priori conditions. Of course, there are many issues that we did not discuss in this work. First, we did not provide a rigorous proof for the convergence of relative distances to the a priori relative distances. Second, we did not answer whether the proposed particle results can be lifted to corresponding results for the corresponding kinetic and hydrodynamic models or not. We leave these interesting issues for a future work. 

\newpage

\appendix

\section{An example for finite-time collision} \label{App-A}
\setcounter{equation}{0}
In this appendix, we provide an example for finite-time collisions to the two-particle system:
\[ N = 2, \quad d = 1, \quad  R_{12}^{\infty} = R_{21}^{\infty} = R. \]
For simplicity, we consider the classical model corresponding to $c=\infty$ so that
\[ v_i=w_i, \quad i =1,2. \] 
In this setting, system \eqref{B-1} becomes
\begin{equation}
\begin{cases}\label{C-4-1}
\displaystyle \frac{dx_1}{dt} = v_1, \quad  \frac{dx_2}{dt} = v_2, \vspace{.2cm} \\
\vspace{0.1cm}
\displaystyle \frac{dv_1}{dt}  =\frac{\kappa_0}{2} \phi(|x_2- x_1|)\left(v_2 - v_1 \right)  \\
\displaystyle \hspace{1cm} + \frac{1}{4} 
	\Big[ \kappa_1 \frac{(v_2 - v_1) (x_2 - x_1)}{ |x_2- x_1|^2}  + \kappa_2\frac{(|x_2- x_1|-R)}{ |x_2- x_1|} \Big] (x_2- x_1), \\
\displaystyle \frac{dv_2}{dt}  =\frac{\kappa_0}{2} \phi(|x_1- x_2|)\left(v_1 - v_2 \right)  \\
\displaystyle \hspace{1cm} + \frac{1}{4} 
	\Big[ \kappa_1 \frac{(v_1 - v_2) (x_1 - x_2)}{ |x_1- x_2|^2}  + \kappa_2\frac{(|x_1- x_2|-R)}{ |x_1- x_2|} \Big] (x_1- x_2).
\end{cases}
\end{equation}
In what follows, we show that there exists initial data leading a finite-time collision. For this, we split its proof into several steps. We first define
\[
	\Phi(x) :=\int_0^x \phi(y) dy, \quad
	\phi_m:=\min\left\{\phi(r):0\leq r\leq \overline{r} = R+\sqrt{\frac{8(\mathcal{E}^c(0)-2)}{\kappa_2}}\right\}.
\]
Then it follows from \eqref{C-2} that $\phi_m$ is a lower bound of $\phi$.
\newline

\noindent $\bullet$~Step A (Basic setting for initial data and coupling strengths):~Consider an initial configuration $\{(x_i^0, v_i^0)\}$ and coupling strengths satisfying the following relations:
\begin{align}
\begin{aligned} \label{C-4-2}
& \mbox{(C1)}:~ |x_1^0| < \frac{R}{2}, \quad |x_2^0|< \frac{R}{2}, \quad x_1^0 < 0 < x_2^0, \quad v_2^0 < 0 < v_1^0, \\
& \mbox{(C2)}:~ \kappa_0 \Phi(x_2^0 - x_1^0) -\kappa_0 \phi_m (x_2^0 - x_1^0)  < v_1^0 - v_2^0, \\
& \mbox{(C3)}:~\kappa_0\Phi(x_2^0-x_1^0)+\frac{\kappa_1}{2}(x_2^0-x_1^0) < v_1^0 - v_2^0, \\
& \mbox{(C4)}:~\kappa_2 \ll 1. 
\end{aligned}
\end{align}
Note that the conditions (C2),(C3) and (C4) can be made by taking
\[    \kappa_0 \ll  |v_1^0 - v_2^0|, \quad  \kappa_1 \ll |v_1^0 - v_2^0|, \quad \kappa_2 \ll 1.  \]
Moreover, for $\kappa_0 = \kappa_1 = \kappa_2 = 0$, the free flow \eqref{C-4-1} with initial data satisfying (C1) leads to a finite-time collision. Thus, the setting (C1) - (C4) can be understood as a perturbative setting for the colliding solution to the free flow.  In the following steps, we will show that some initial configuration satisfying the relations in \eqref{C-4-2} will lead to a finite time collision along the dynamics \eqref{C-4-1} using a contradiction argument by adjusting $\kappa_1$ and $\kappa_2$.  \newline

Suppose that the finite-time collision does not occur for initial data \eqref{C-4-2}, i.e., the solution is globally well-posed and satisfies 
\begin{equation} \label{C-4-2-1}
x_1(t)<x_2(t), \quad t \geq 0. 
\end{equation}
We will see that this relation leads to a contradiction by deriving a differential inequality for $x_2 - x_1$.  \newline

\noindent $\bullet$~Step B (Dynamics for $v_2 - v_1$): we use \eqref{C-4-1} to see that $v_2 - v_1$ satisfies 
\begin{align}
\begin{aligned} \label{C-4-3}
\frac{d}{dt} (v_2 - v_1) &= - \Big(  \kappa_0 \phi(x_2 - x_1) + \frac{\kappa_1}{2}   \Big) (v_2 - v_1)  -\frac{\kappa_2}{2} \Big( (x_2 - x_1) - R \Big) \\
&= -\frac{d}{dt} \Big(  \kappa_0 \Phi(x_2 - x_1) + \frac{\kappa_1}{2} (x_2 - x_1)  \Big)  -\frac{\kappa_2}{2} \Big( (x_2 - x_1) - R \Big).
\end{aligned}
\end{align}

\noindent $\bullet$~Step C (Dynamics for $x_2 - x_1$): By \eqref{C-4-1} and \eqref{C-4-3}, one has 
\begin{align}
\begin{aligned} \label{C-4-4}
	\frac{d}{dt}(x_2-x_1) &= v_2-v_1 \\
	& = v_2^0 - v_1^0 + \int_{0}^{t} \frac{d}{dt}(v_2(s) - v_1(s)) ds\\
	&=-(x_2-x_1)\left(\frac{\kappa_0 \Phi(x_2-x_1)}{x_2-x_1}+\frac{\kappa_1}{2}\right)-\frac{\kappa_2}{2}\int_0^t\big(x_2(s)-x_1(s)-R\big)ds\\
	&\hspace{0.4cm}+\left(v_2^0-v_1^0+\kappa_0\Phi(x_2^0-x_1^0)+\frac{\kappa_1}{2}(x_2^0-x_1^0)\right)\\
	&\le-(x_2-x_1)\left(\kappa_0 \phi_m+\frac{\kappa_1}{2}\right)+2 \sqrt{\kappa_2\mathcal{E}^c(0)}t \\
	&\hspace{.4cm}+\left(v_2^0-v_1^0+\kappa_0\Phi(x_2^0-x_1^0)+\frac{\kappa_1}{2}(x_2^0-x_1^0)\right)\\
	&=: -A(x_2-x_1) + Bt + C,
\end{aligned}
\end{align}
where we used \eqref{B-1-2} and Proposition \ref{P2.1}:
\begin{align*}
 \frac{\kappa_2}{16} (x_2 - x_1 - R)^2 \leq {\mathcal E}_p^c(0) \leq  {\mathcal E}^c(0), \quad \mbox{i.e.,} \quad  |x_2 - x_1 -R| \leq 4 \sqrt{ \frac{{\mathcal E}^c(0)}{\kappa_2}}. 
\end{align*}
Moreover, we also use (C3) to see
\begin{equation} \label{C-4-6}
A > 0, \quad B > 0, \quad C < 0.
\end{equation}
 
 \vspace{0.2cm}
 
\noindent $\bullet$~Step D (Behavior of $x_2-x_1$): Consider the Cauchy problem to the following ODE:
\begin{equation*} 
\begin{cases} 
\displaystyle y'(t)= -A y + Bt + C, \quad t > 0, \\
\displaystyle y(0)=x_2^0-x_1^0 > 0.
\end{cases}
\end{equation*}
This yields
\begin{equation} \label{C-5}
	y(t)
	= -\frac{B}{A^2}+\frac{B}{A}t+\frac{C}{A}+ k e^{-At}, \quad  k=x_2^0-x_1^0+\frac{B-AC}{A^2}.
\end{equation}
Then, by comparison principle for ODE, we have
\begin{equation} \label{C-5-1}
0 < x_2(t) - x_1(t) \leq y(t), \quad t > 0.
\end{equation}
By direct calculation, the solution$y$ in \eqref{C-5} satisfies 
\[ y^{\prime}(t) = \frac{B}{A} - k A e^{-At}, \quad y^{\prime \prime}(t) =  kA^2 e^{-At}. \]
This yields
\begin{equation} \label{C-5-2}
y^{\prime}~\mbox{is increasing}, \quad y^{\prime}(0) =  -\frac{k A^2 - B}{A}, \quad  \lim_{t \to \infty} y^{\prime}(t)  = \frac{B}{A} > 0 \quad \text{and} \quad y(0)  > 0.    
\end{equation}
Note that the condition (C2) in \eqref{C-4-2} is equivalent to $k A^2 - B > 0$:
\begin{align}
\begin{aligned} \label{C-5-3}
k A^2 - B > 0 \quad &\Longleftrightarrow \quad A \Big( (x_2^0 - x_1^0) A - C  \Big) > 0 \\
&\Longleftrightarrow \quad A \Big( \kappa_0 \phi_m (x_2^0 - x_1^0) - (v_2^0 - v_1^0) -\kappa_0 \Phi(x_2^0 - x_1^0) \Big)   > 0    \\
 &\Longleftrightarrow \quad  \kappa_0 \Phi(x_2^0 - x_1^0) -\kappa_0 \phi_m (x_2^0 - x_1^0)  < v_1^0 - v_2^0,
\end{aligned}
\end{align} 
where we used $A > 0$.  On the other hand, note that 
\begin{equation} \label{C-5-4}
y^{\prime}(t_*) = 0  \quad  \Longleftrightarrow \quad   k A e^{-At_*} = \frac{B}{A} \quad \Longleftrightarrow \quad  t_*  = -\frac{1}{A} \ln \Big( \frac{B}{kA^2} \Big).
\end{equation}
Therefore, it follows from \eqref{C-5-2}, \eqref{C-5-3} and \eqref{C-5-4} that 
\begin{align}
\begin{aligned} \label{C-5-5}
& \kappa_0 \Phi(x_2^0 - x_1^0) -\kappa_0 \phi_m (x_2^0 - x_1^0)  < v_1^0 - v_2^0  \\
& \hspace{1cm} \Longleftrightarrow \quad t_* > 0 \quad \mbox{and} \quad   y(t) \geq y(t_*) =  \frac{B}{A^2}\log\frac{A^2k}{B}+\frac{C}{A} > 0 \quad \forall~t \geq 0.
\end{aligned}
\end{align}

\vspace{0.2cm}

 \noindent $\bullet$~Step E (Behavior of $y(t_*)$ for $\kappa_2 \ll 1$): we use \eqref{C-4-4} and \eqref{C-5} to find 
\[  \lim_{\kappa_2 \to 0} B = 0, \quad \lim_{\kappa_2 \to 0} k = x_2^0-x_1^0 - \frac{C}{A}. \]
This and \eqref{C-5-5} yield
\[ \lim_{\kappa_2 \to 0} y(t_*) = \frac{C}{A}  > 0, \quad \mbox{i.e.,} \quad C > 0.  \]
Thus, for $\kappa_2 \ll 1$, 
\[ C > 0, \] 
which is contradictory to $\eqref{C-4-6}_3$. Finally, the relation \eqref{C-4-2-1} can not hold for all $t$ and we will have a finite-time collision. 
\section{Proof of Lemma \ref{L4.1}} \label{App-B}
\setcounter{equation}{0}
In this appendix, we provide a proof of Lemma \ref{L4.1}. \newline

\noindent (1)~ Since 
\[
f_y(x):=\frac{x}{\sqrt{x^2-y^2}} \quad \mbox{and} \quad g_y(x):=x^2\left(\frac{x}{\sqrt{x^2-y^2} }- 1 \right), \quad (x>y>0)
\]
are strictly decreasing functions of $x$ and $F_i^{c'} \geq 1$ for any ${c'} \in [c,\infty]$, we have
\begin{align*}
	&\Gamma(\bv_i^{c'}(0)) = \tilde{\Gamma}^{F_i^{c'} c'}(\bw_i^0) 
	\leq \tilde{\Gamma}^{c'}(\bw_i^0) < \tilde{\Gamma}^{c}(\bw_i^0) < \infty, \\
	&(c')^2(\tilde{\Gamma}^{c'}(\bw_i^0)-1) \leq c^2(\tilde{\Gamma}^c(\bw_i^0)-1) < \infty,
\end{align*}
where $\tilde{\Gamma}_i^{c}(\bw_i^0)$ is finite because $|\bw_i^0|<c$. Thus, we have
\begin{align*}
	\sup_{c' \in [c,\infty]} \mathcal{E}^{c'}_{k}(0) 
	&=\sup_{c' \in [c,\infty]} \sum_{i=1}^{N} \left( (c')^2(\Gamma^{c'}(\bv_i^{c'}(0))-1)+(({\Gamma}^{c'}(\bv_i^{c'}(0)))^2-\log{\Gamma^{c'}(\bv_i^{c'}(0))}) \right) \\
	&\leq \sup_{c' \in [c,\infty]} \sum_{i=1}^{N} \left( (c')^2(\tilde{\Gamma}^{c'}(\bw^0_i)-1)+(({\tilde{\Gamma}}^{c'}(\bw_i^0))^2-\log{\tilde{\Gamma}^{c'}(\bw_i^0)}) \right) \\
	&\leq \sum_{i=1}^{N} \left( c^2(\tilde{\Gamma}^{c}(\bw_i^0)-1)+(({\tilde{\Gamma}}^{c}(\bw_i^0))^2-\log{\tilde{\Gamma}^{c}(\bw_i^0)}) \right) 
	= \tilde{\mathcal{E}}^c_k(0).
\end{align*}
Since initial potential energy is $c$-independent, and the total energy is nonincreasing, we obtain
\begin{align}
\begin{aligned}\label{D-8}
	\sup_{\substack{c' \in [c,\infty] \\ t \in [0,\tau)}} \mathcal{E}^{c'}_{k}(t) &\le
	\sup_{\substack{c' \in [c,\infty] \\ t \in [0,\tau)}} \left( \mathcal{E}^{c'}_{k}(t) + \mathcal{E}_{p}^{c'}(t) \right)
	\\
	&= \sup_{c' \in [c,\infty]} \left( \mathcal{E}^{c'}_{k}(0) + \mathcal{E}^{c'}_{p}(0) \right)
	= \tilde{\mathcal{E}}^{c}_{k}(0) + \mathcal{E}^{c}_{p}(0) < \infty.
\end{aligned}\end{align}
Now, suppose on the contrary that
\[
	\sup_{\substack{c' \in [c,\infty) \\ t \in [0,\tau)}}\max_{i}\frac{\left |\bv^{c'}_i(t)\right |}{c'} = 1
\]
holds. Then there exist an index $\alpha$ and a sequence $(c_n,t_n)$ in $[c,\infty) \times [0,\tau)$ such that for all $k \in (0,1)$, existence of $N=N(k) \in \mathbb{N}$ satisfying
\begin{align}\label{D-9}
	\frac{\left |\bv^{c_n}_\alpha(t_n)\right|}{c_n} > k \quad \text{whenever} \quad n>N
\end{align}
is guaranteed. However, if \eqref{D-9}  is satisfied, then we have
\[
	\Gamma^{c_n}(\bv^{c_n}_\alpha(t_n)) = \frac{1}{\sqrt{1-\left(\frac{|\bv_\alpha^{c_n}(t_n)|}{c_n}\right)^2}} > \frac{1}{\sqrt{1-k^2}}.
\]
Since $k$ can be taken arbitrary in $(0,1)$, $\Gamma^{c_n}(\bv^{c_n}_\alpha(t_n))$ can be arbitrary large, and so is ${\mathcal{E}}^{c_n}_k(t_n)$. Hence for sufficiently large $m$, we have 
\[{\mathcal{E}}^{c_m}_k(t_m) > \tilde{\mathcal{E}}^{c}_{k}(0) + \mathcal{E}^c_{p}(0).\]
Then \eqref{D-8} yields
\[
	\mathcal{E}^{c_m}_k(t_m) \leq \sup_{\substack{c' \in [c,\infty] \\ t \in [0,\tau)}}\mathcal{E}^{c'}_k(t)
	\leq \tilde{\mathcal{E}}^{c}_{k}(0) + \mathcal{E}^c_{p}(0) < \mathcal{E}^{c_m}_k(t_m),
\]
which gives a contradiction. Therefore we verify \eqref{D-5}. \vspace{.2cm}\\

\noindent $\diamond$ (Proof of \eqref{D-6}) 
By the same estimates as in \eqref{C-4}, we have 
\begin{align*}
	1+\frac{1}{2}\sup_{\substack{ c' \in [c, \infty] \\ t \in [0,\tau)}} |\bv_i^{c'}(t)|^2
	\leq \sup_{ c' \in [c, \infty] }\mathcal{E}^{c'}_k(0) + \mathcal{E}^c_p(0)
	\leq \tilde{\mathcal{E}}^{c}_k(0) + \mathcal{E}^c_p(0) < \infty.
\end{align*}
Then, together with \eqref{D-5}, we have
\[
	\sup_{t \in [0,\tau)} |\bv_i^{c'}(t)|
	\leq \min\left\{ \delta c' , \sqrt{2(\tilde{\mathcal{E}}^{c}_k(0) + \mathcal{E}^c_p(0) -1)} \right\}
	< \infty.
\]
On the other hand, for unit speed vector $\bu \in B_1(\bf{0})$, we have $\Gamma^{c'}({\delta {c'} \bu})=\frac{1}{\sqrt{1-\delta^2}}$ and thus
\[
	| \hat{w}_{c'} (\delta c' \bu ) | = | F^{c'} (\delta {c'} \bu) \delta c'\bu | = \frac{\delta c'}{\sqrt{1-\delta^2}} + \frac{\delta}{c'({1-\delta^2)}}
	=: w_{c'}^\delta.
\]
Therefore, if we define
\[
	c'' :=
	\begin{cases}
	\displaystyle \frac{\sqrt{2(\tilde{\mathcal{E}}^{c}_k(0) + \mathcal{E}^c_p(0) -1)}}{\delta} \quad &\text{if} \quad \sqrt{2(\tilde{\mathcal{E}}^{c}_k(0) + \mathcal{E}^c_p(0) -1)} \le \delta c' \vspace{.2cm} \\
	\displaystyle c' \quad &\text{if} \quad \sqrt{2(\tilde{\mathcal{E}}^{c}_k(0) + \mathcal{E}^c_p(0) -1)} > \delta c',
	\end{cases}
\]
then since $c''$ is finite and $\bv \mapsto | \hat{w}^{c'}(\bv) |$ is a decreasing function of $c' \in [c,\infty]$, we have
\[
	\sup_{\substack{c' \in [c,\infty] \\ t \in [0,\tau)}}|\bw_i^{c'}(t) |
	\leq \sup_{c' \in [c,c'']} w^\delta_{c'}~ =: U_w
	< \infty,
\] which proves \eqref{D-6}. \vspace{.2cm}


\noindent (2)~By the same procedure as in the proof of Corollary \ref{L3.1} but $\mathcal{E}^c(0)=\mathcal{E}^c_k(0)+\mathcal{E}^c_p(0)$ is replaced by $\tilde{\mathcal{E}}_k^c(0)+\mathcal{E}^c_p(0)$, it is straightforward to see
\[ \label{D-10}
 \min_{k \neq l}R^\infty_{kl}-\sqrt{\frac{4N(\tilde{\mathcal{E}}_k^c(0)+\mathcal{E}^c_p(0)-N)}{\kappa_2}}  \leq r^{c'}_{ij}(t) \leq \max_{k \neq l}R^\infty_{kl}+\sqrt{\frac{4N(\tilde{\mathcal{E}}_k^c(0)+\mathcal{E}^c_p(0)-N)}{\kappa_2}},
\]
for arbitrary $c' \in [c,\infty]$. Then the leftmost term is positive if and only if \eqref{D-7} holds. $\qed$

\section{Derivation of Gronwall's inequality \eqref{D-7-2}}  \label{App-C}
\setcounter{equation}{0} 

In this appendix, we provide the derivation of the differential inequality \eqref{D-7-2}. We begin with a following elementary lemma.

\begin{lemma}\label{L4.2}
For nonzero $\bx,\by \in \bbr^d$, we have
\[
\left|  \frac{\bx}{|\bx|} - \frac{\by}{|\by|} \right | \leq \frac{2 |\bx-\by|}{\mathrm{min}( |\bx|, |\by|)}.
\]
\end{lemma}
\begin{proof} Note that 
\begin{align*}
\begin{aligned}
 \Big| \frac{\bx}{ | \bx |} -  \frac{\by}{| \by|} \Big| &=  \Big| \frac{\bx - \by}{| \bx |} + \frac{  (| \by | - | \bx |) \by }{ | \bx | \cdot | \by |} \Big| \\
 &\leq   \frac{1}{\mathrm{min}(|\bx|, |\by|)} \Big(
| \bx - \by |  + | | \by | - | \bx| | \Big) \leq  \frac{2 | \bx - \by |}{\mathrm{min}(|\bx|,~|\by|)}.
\end{aligned}
\end{align*}
\end{proof}

\noindent $\bullet$~Step A: By \eqref{B-3}, one has 
\begin{align}\label{D-7-3}
|\bv_i^c - \bw_i^\infty| \leq |\bv_i^c - \bw_i^c|  + |\bw_i^c - \bw_i^\infty|  \leq |\bw_i^c - \bw_i^\infty| + \mathcal{O}(c^{-2}),
\end{align}
where the $\mathcal{O}(c^{-2})$ term is independent of $c'$ and $t$. 

\vspace{0.5cm}

\noindent $\bullet$~Step B (Estimate of $|\bx_i^c(t)-\bx_i^\infty(t)|^2$): we use \eqref{D-7-3} to get 
\begin{align*}
\begin{aligned} \label{D-7-3-1}
	&\frac{d}{dt} \sum_{i=1}^{N} |\bx_i^c-\bx_i^\infty|^2 = 2  \sum_{i=1}^{N} \left\langle \bx_i^c-\bx_i^\infty , \frac{d}{dt}(\bx_i^c-\bx_i^\infty) \right\rangle \leq 2  \sum_{i=1}^{N} |\bx_i^c-\bx_i^\infty | \cdot |\bv_i^c-\bw_i^\infty| \\
	&\hspace{0.2cm} \leq  \sum_{i=1}^{N} \Big( |\bx_i^c-\bx_i^\infty|^2 + |\bv_i^c-\bw_i^\infty|^2 \Big) \leq  \sum_{i=1}^{N} \Big( |\bx_i^c-\bx_i^\infty|^2 + 2 |\bw_i^c - \bw_i^\infty|^2 \Big) + \mathcal{O}(N c^{-4}),
\end{aligned}
\end{align*}
where we used $(|a|+|b|)^2 \leq 2(|a|^2+|b|^2)$. 

\vspace{0.5cm}

\noindent $\bullet$~Step C (Estimate of $|\bw_i^c(t)-\bw_i^\infty(t)|^2$): It follows from \eqref{D-1} and \eqref{D-2} that 
\begin{equation*} \label{D-7-4}
\frac{d}{dt} (\bw_i^c - \bw_i^{\infty}) =   (\mathcal{I}_i^c - \mathcal{I}_i^\infty ) +( \mathcal{J}_i^c - \mathcal{J}_i^\infty ) + (\mathcal{K}_i^c -  \mathcal{K}_i^\infty). 
\end{equation*}
This yields
\begin{align*}
\begin{aligned} \label{D-7-4-1}
	\frac{d}{dt} \sum_{i=1}^{N} |\bw_i^c-\bw_i^\infty|^2  &\leq 2 \sum_{i=1}^{N} |\bw_i^c-\bw_i^\infty |\frac{d}{dt} |\bw_i^c-\bw_i^\infty| \\
	& \leq 2 \sum_{i=1}^{N} |\bw_i^c-\bw_i^\infty| \Big(| \mathcal{I}_i^c - \mathcal{I}_i^\infty| + | \mathcal{J}_i^c - \mathcal{J}_i^\infty| + | \mathcal{K}_i^c - \mathcal{K}_i^\infty | \Big ).
\end{aligned}
\end{align*}
In what follows, we estimate the following terms separately:
\[ | \mathcal{I}_i^c - \mathcal{I}_i^\infty|, \quad  | \mathcal{J}_i^c - \mathcal{J}_i^\infty|, \quad | \mathcal{K}_i^c - \mathcal{K}_i^\infty|. \]

\vspace{0.2cm}

\noindent $\diamond$ Case A: Note that, thanks to Lemma \ref{L3.1}, $\phi$ can be regarded as a function defined on compact interval $[\underline{r}, \overline{r}]$, and therefore $\phi$ is a Lipschitz continuous function. We use 
\[ \max_{i,j} |\bw^{c'}_i - \bw^{c'}_j | \leq 2 U_w, \quad  \phi(r^\infty_{ji}) \leq \phi_M \]
to find
\begin{align}
\begin{aligned} \label{D-7-5}
	&| \mathcal{I}_i^c - \mathcal{I}_i^\infty| \\
	& ~ \leq \frac{\kappa_0}{N} \sum_{j=1}^N \left[ |\phi(r^c_{ji})-\phi(r^\infty_{ji})| |\bv_j^c -\bv_i^c| 
	+ \phi(r_{ji}^\infty)( |\bv_i^c -\bw_i^c|+ |\bv_j^c-\bw_j^\infty|) \right] \\
	& ~ \leq \frac{\kappa_0}{N} \sum_{j=1}^N \left[ |\phi(r^c_{ji})-\phi(r^\infty_{ji})| |\bw_j^c -\bw_i^c| 
	+ \phi(r_{ji}^\infty)(|\bw_j^c-\bw_j^\infty| + \mathcal{O}(c^{-2})) \right] \\
	& ~ \leq \frac{2\kappa_0 U_w[\phi]_{\text{Lip}}}{N} \sum_{j=1}^N (|\bx_i^\infty-\bx_i^c| + |\bx_j^\infty - \bx_j^c|)
	+ \frac{\kappa_0\phi_M}{N} \sum_{j=1}^N |\bw_j^c - \bw_j^\infty| + \mathcal{O}(c^{-2}) \\
	&~\leq 2\kappa_0 U_w [\phi]_{\text{Lip}} |\bx_i^\infty-\bx_i^c| + \frac{2\kappa_0 U_w[\phi]_{\text{Lip}}}{N} \sum_{j=1}^N |\bx_j^\infty - \bx_j^c| + \frac{\kappa_0\phi_M}{N} \sum_{j=1}^N |\bw_j^c - \bw_j^\infty| + \mathcal{O}(c^{-2}),
\end{aligned}
\end{align}
where $[\phi]_{\text{Lip}}$ is the Lipchitz constant of $\phi$.

\vspace{0.2cm}

\noindent $\diamond$ Case B: It follows from Lemma \ref{L2.2} and Theorem \ref{T3.1} that there exists a positive constant $r_m$ such that 
\begin{equation} \label{D-7-6}
 0 <  r_m \leq \inf_{t > 0} \min_{i,j} | \bx_i(t) - \bx_j(t)|.
 \end{equation}
where $r_m$ is uniform in $t$ and $c' \in [c,\infty]$. \newline

\noindent Note that 
\begin{align}
\begin{aligned} \label{D-7-6}
	&| \mathcal{J}_i^c - \mathcal{J}_i^\infty| \\
	& \hspace{0.2cm}  \leq ~ \frac{\kappa_1}{2N} \Bigg | \sum_{j=1}^N \Bigg[
	\left( \frac{\br_{ji}^c}{r_{ji}^c} - \frac{\br_{ji}^\infty}{r_{ji}^\infty} \right) \left\langle \bv_j^c-\bv_i^c, \frac{\br_{ji}^c}{r_{ji}^c} \right\rangle + \frac{\br_{ji}^\infty}{r_{ji}^\infty}\left( \left\langle \bv_j^c-\bv_i^c, \frac{\br_{ji}^c}{r_{ji}^c} \right\rangle
	- \left\langle \bw_j^\infty-\bw_i^\infty, \frac{\br_{ji}^\infty}{r_{ji}^\infty} \right\rangle \right)
	\Bigg] \Bigg | \\
	& \hspace{0.2cm} =: \Delta_{i,1}^{c, \infty}({\mathcal J}) + \Delta_{i,2}^{c, \infty}({\mathcal J}).
\end{aligned}
\end{align}
Below, we estimate the term $\Delta_{i,j}^{c, \infty}({\mathcal J})$ one by one. \newline

\noindent~$\diamond$~(Estimate on $\Delta_{i,1}^{c, \infty}({\mathcal J})$): By Lemma \ref{L2.2}, one has 
\[
| \bv_j^c-\bv_i^c| \leq |  \bw_j^c-\bw_i^c | \leq 2 U_w.
\] 
We use the above estimate to obtain
\begin{equation} \label{D-7-7}
\Delta_{i,1}^{c, \infty}({\mathcal J}) \leq  \frac{\kappa_1 U_w}{N r_m}  \sum_{j=1}^N |\bx_j^c - \bx_j^\infty| + \frac{\kappa_1 U_w}{r_m} |\bx_i^c - \bx_i^\infty|.
\end{equation}

\vspace{0.2cm}

\noindent~$\diamond$~(Estimate on $\Delta_{i,2}^{c, \infty}({\mathcal J})$): By direct calculation, we have
\begin{align}
\begin{aligned} \label{D-7-8}
&\left | \left\langle \bv_j^c-\bv_i^c, \frac{\br_{ji}^c}{r_{ji}^c} \right\rangle
	- \left\langle \bw_j^\infty-\bw_i^\infty, \frac{\br_{ji}^\infty}{r_{ji}^\infty} \right\rangle \right | \\
& \hspace{0.5cm} \leq \left | \left\langle \bv_j^c-\bv_i^c, \frac{\br_{ji}^c}{r_{ji}^c} - \frac{\br_{ji}^\infty}{r_{ji}^\infty} \right\rangle
	- \left\langle \bw_j^\infty - \bv_j^c + \bv_i^c - \bw_i^\infty, \frac{\br_{ji}^\infty}{r_{ji}^\infty} \right\rangle \right | \\
& \hspace{0.5cm} \leq \frac{2U_w}{r_m} \Big ( |\bx_j^c - \bx_j^\infty | + |\bx_i^c - \bx_i^\infty| \Big ) + |\bw_j^c - \bw_j^\infty| + |\bw_i^c - \bw_i^\infty|.
\end{aligned}
\end{align}
In \eqref{D-7-6}, we combine \eqref{D-7-7} and \eqref{D-7-8} to find 
\begin{align}
\begin{aligned} \label{D-7-9}
| \mathcal{J}_i^c - \mathcal{J}_i^\infty | &\leq  C \left(  \sum_{j=1}^N |\bx_j^c - \bx_j^\infty| +  |\bx_i^c - \bx_i^\infty| 
+ \sum_{j=1}^N |\bw_j^c - \bw_j^\infty|  +  |\bw_i^c - \bw_i^\infty|\right),
\end{aligned}
\end{align}
where $C$ is a generic positive constant independent of $t$ and $c' \in [c,\infty]$.

\vspace{0.2cm}

\noindent $\diamond$ Case C: we use the uniform spatial boundedness from Theorem \ref{T3.2} that 
\begin{equation} \label{D-7-10}
	|\mathcal{K}_i^\infty - \mathcal{K}_i^c | 
	\leq C\kappa_2 |\bx_i^\infty-\bx_i^c| +  \frac{C\kappa_2}{N}\sum_{j=1}^N |\bx_j^\infty-\bx_j^c|,
\end{equation}
where $C$ is a positive constant independent of $t$ and $c' \in [c,\infty]$. \newline

\noindent Finally, we use 
\begin{align*}
	\Big| \frac{d}{dt} |\bw_i^c-\bw_i^\infty|^2 \Big|
	&\leq 2 |\bw_i^c-\bw_i^\infty | \Big| \frac{d}{dt} |\bw_i^c-\bw_i^\infty|  \Big| \\
	&\leq |\bw_i^c-\bw_i^\infty| (| \mathcal{I}_i^c - \mathcal{I}_i^\infty| + | \mathcal{J}_i^c - \mathcal{J}_i^\infty| + | \mathcal{K}_i^c - \mathcal{K}_i^\infty| )
\end{align*}
together with
\[
|\bw_i^c-\bw_i^\infty | \cdot  |\bx_i^c-\bx_i^\infty| \leq \frac{1}{2}(|\bw_i^c-\bw_i^\infty|^2 + |\bx_i^c-\bx_i^\infty|^2),
\]
and collect all the estimates \eqref{D-7-5}, \eqref{D-7-9} and \eqref{D-7-10} to derive \eqref{D-7-2}.


\section{Proof of Lemma \ref{L5.1}} \label{App-D}
\setcounter{equation}{0}
\noindent (1)~ The proof of the first assertion is as follows: let $\gamma$ be a geodesic curve connecting $\bx$ to $\by$. Then, geodesic equation $\nabla_{\dot{\gamma}}{\dot{\gamma}}=0$ induces that a tangent vector $\bv \in T_{\bx}\mathcal{M}$ maps to $\bw \in T_{\by}\mathcal{M}$ by a parallel transport operator along $\gamma$. Here, since $\bv$ is a multiple of $-\log_{\bx}\by$ and $\bw$ is a multiple of $-\log_{\by}\bx$, we have the desired estimate because
\[\|\log_{\bx}\by\|_{\bx}=\|\log_{\by}\bx\|_{\by}=d(\bx,\by),\]
where we used a following property: $\|\log_{\bx}\by\|_{\bx}=d(\bx,\by)$.\\
\newline
\noindent (2)~ Consider smooth geodesics $\gamma_1,\gamma_2:(-\varepsilon,\varepsilon)\rightarrow \mathcal{M}$ such that
\[ \gamma_1(0)=\bx, \quad  \dot{\gamma}_1(0)=\bv \in T_{\bx}\mathcal{M}, \quad \gamma_2(0)=\by, \quad  \dot{\gamma}_2(0)=\bw \in T_{\by}\mathcal{M}. \] 
Moreover, let $\tilde{\gamma}: [0,1]\to \mathcal{M}$ be a length-minimizing geodesic curve connecting $\bx$ to $\by$. Then, we admit a homotopy $\Phi$ which is the variation of $\tilde{\gamma}$ denoted by
\[\Phi:=(-\varepsilon,\varepsilon)\times [0,1] \to \mathcal{M},\quad \Phi(t,s)=\tilde{\gamma}_{t}(s),\] where $\tilde{\gamma}_{t}(s):[0,1]\to \mathcal{M}$ is a geodesic curve connecting $\gamma_1(t)$ to $\gamma_2(t)$. Thus, we can get
\[ \label{Ap-1}
d(\gamma_1(t),\gamma_2(t))=\int_{0}^{1} \sqrt{g_{{\tilde{\gamma}}_{t}(s)} \left(\frac{d{\tilde{\gamma}}_{t}(s)}{ds},\frac{d{\tilde{\gamma}}_{t}(s)}{ds} \right)} ds,\qquad \frac{D{\tilde{\gamma}}(s)}{ds}=0,
\] 
where we used the relation:
\[  \frac{D{\tilde{\gamma}}(s)}{ds}:=\nabla_{\dot{\tilde{\gamma}}(s)}\dot{\tilde{\gamma}}(s). \]
This implies 
\[\frac{d}{ds}g_{{\tilde{\gamma}}_{t}(s)}\left(\dot{\tilde{\gamma}}_{t}(s), \dot{\tilde{\gamma}}_{t}(s)\right)=0.\] Hence,
\[d^2(\gamma_1(t),\gamma_2(t))=\int_{0}^{1} g_{{\tilde{\gamma}}_{t}(s)} \left(\dot{\tilde{\gamma}}_{t}(s), \dot{\tilde{\gamma}}_{t}(s)\right) ds = \int_{0}^{1} g_{{\tilde{\gamma}}_{t}(s)}\left(\frac{\partial \Phi(t,s)}{\partial s}, \frac{\partial \Phi(t,s)}{\partial s}\right) ds\] because 
$g\left(\dot{\tilde{\gamma}}_{t}(s), \dot{\tilde{\gamma}}_{t}(s)\right)_{{\tilde{\gamma}}_{t}(s)}$ is a constant. Moreover,  we observe that
\begin{align}\label{Ap-2}
\begin{aligned}
\frac{d}{dt}(d^2(\gamma_1(t),\gamma_2(t)))&=2\int_{0}^{1} g_{{\tilde{\gamma}}_{t}(s)}\left(\frac{D}{\partial t}\frac{\partial \Phi(t,s)}{\partial s}, \frac{\partial \Phi(t,s)}{\partial s}\right) ds\\
&=2\int_{0}^{1} g_{{\tilde{\gamma}}_{t}(s)}\left(\frac{D}{\partial s}\frac{\partial \Phi(t,s)}{\partial t}, \frac{\partial \Phi(t,s)}{\partial s}\right) ds,
\end{aligned}
\end{align}where we used Lemma $3.4$ in {\cite{C}} in the second equation that is an invariance with respect to an interchange of the order between covariant derivative and time-derivative. It follows from \eqref{Ap-2} that
\begin{align}\label{D.3}
\begin{aligned}
&\frac{d}{dt}(d^2(\gamma_1(t),\gamma_2(t))) =2\int_{0}^1 \frac{d}{d s}g_{{\tilde{\gamma}}_{t}(s)} \left(\frac{\partial \Phi(t,s)}{\partial t}, \frac{\partial \Phi(t,s)}{\partial s}\right) ds\\
& \hspace{1cm} =2 g_{{\tilde{\gamma}}_{t}(1)}\left(\frac{\partial \Phi(t,1)}{\partial t}, \frac{\partial \Phi(t,s)}{\partial s} \Big|_{s=1}\right) -2 g_{{\tilde{\gamma}}_{t}(0)} \left(\frac{\partial \Phi(t,0)}{\partial t}, \frac{\partial \Phi(t,s)}{\partial s} \Big|_{s=0}\right).
\end{aligned}
\end{align}
Here, we use the facts that $\Phi(0,s)=\tilde{\gamma}(s)$ is the length-minimizing geodesic and $\Phi(t,0)={\gamma_1}(t)$ to obtain
\begin{align}\label{Ap-3}
\frac{\partial}{\partial t}\Phi(t,0)\Big|_{t=0}=\dot{{\gamma_1}}(0)=\bv,\quad\frac{\partial}{\partial s}\Phi(0,s) \Big|_{s=0}=\dot{\tilde{\gamma}}(0)=\log_{\bx}\by.
\end{align}In addition, we use that $\Phi(t,1)={\gamma_2}(t)$ is the length-minimizing geodesic and $\Phi(0,s)=\tilde{\gamma}(s)$ to get

\begin{align}\label{Ap-4}
\frac{\partial}{\partial t}\Phi(t,1) \Big|_{t=0}=\dot{{\gamma_2}}(0)=\bw,\quad\frac{\partial}{\partial s}\Phi(0,s) \Big |_{s=1}=\dot{\tilde{\gamma}}(1)=-\log_{\by}\bx
\end{align}
Finally, we combine \eqref{Ap-3} and \eqref{Ap-4} with \eqref{D.3} to lead to the following result:
\begin{align*}
\frac{d}{dt}(d^2(\gamma_1(t),{\gamma_2}(t))) \Big|_{t=0}&=2g_{\scriptsize {\by}}(\bw,-\log_{\by}\bx) -2g_{\scriptsize{\bx}}\left(\bv,\log_{\bx}\by\right) \\
&=2g_{\scriptsize{\bx}}(P_{\bx\by}\bw -\bv,\log_{\bx}{\by}),
\end{align*}
where we used the first assertion to the last equation. Finally, we complete the proof of the second assertion.

\section{Derivation of a bonding force} \label{App-E}
\setcounter{equation}{0}
For the inter-particle bonding force ${\bf F}_{ij}$ between the $i$-th and $j$-th particles, we set 
\begin{equation} \label{Ap-5}
 {\bf F}
_{ij} = a_{ij} \frac{\log_{\bx_i}\bx_j}{d(\bx_i,\bx_j)}, 
\end{equation} 
and we will define $a_{ij}$ in such a way that geodesic distance $d_{ij}$ relaxes to the a priori target distance $R_{ij}^{\infty}$.  Let $R^\infty_{ij}$ be a strictly positive constant depending on indices $(i,j)\in\{(i,j)\}_{1\leq i, j \leq N}$. Next, we set
\[e_{ij}:= R^\infty_{ij}-d_{ij}.\] 
This yields
\begin{align}\label{E-4}
\dot{e}_{ij}=-\dot{d}_{ij}=-\tilde{\bv}_{ij}-\tilde{\bv}_{ji},\quad \ddot{e}_{ij}=-\dot{\tilde{\bv}}_{ij}-\dot{\tilde{\bv}}_{ji},
\end{align}
where ${\tilde{v}}_{ij}$ is the velocity component of $\bv_i$ along with $-\log_{\bx_i}\bx_j$:
\begin{align}\label{E-5}
{\bv}_{ij}=\frac{g_{\scriptsize{\bx_i}}(\bv_i,-\log_{\bx_i}\bx_j)}{d_{ij}},\quad {\bv}_{ji}=\frac{g_{{\scriptsize \bx_j}}(\bv_j,-\log_{\bx_j}\bx_i)}{d_{ij}},
\end{align} where we used Lemma \ref{L5.1} to derive 
\begin{align*}
{\bv}_{ij}+ {\bv}_{ji}&=\frac{g_{{\scriptsize \bx_i}}(\bv_i,-\log_{\bx_i}\bx_j)}{d_{ij}}+\frac{g_{{\scriptsize \bx_j}}(\bv_j,-\log_{\bx_j}\bx_i)}{d_{ij}}\\
&=\frac{g_{{\scriptsize \bx_i}}(\bv_i,-\log_{\bx_i}\bx_j)}{d_{ij}}+\frac{g_{{\scriptsize \bx_i}}(P_{ij}\bv_j,P_{ij}-\log_{\bx_j}\bx_i)}{d_{ij}}\\
&=\frac{g_{{\scriptsize \bx_i}}(\bv_i,-\log_{\bx_i}\bx_j)}{d_{ij}}-\frac{g_{{\scriptsize \bx_i}}(P_{ij}\bv_j,-\log_{\bx_i}\bx_j)}{d_{ij}}\\
&=-\frac{g_{{\scriptsize \bx_i}}(P_{ij}\bv_j-\bv_i,-\log_{\bx_i}\bx_j)}{d_{ij}} =\dot{d}_{ij}.
\end{align*} 
Consider a pairwise bonding-interaction between two particles using acceleration $a_{ij}$ which is a bonding force of $i$-th particle encouraged from $j$-th particle. By the action and reaction's principle, we set $a_{ij}=a_{ji}$ and obtain 
\[\ddot{e}_{ij}=-\dot{\tilde{\bv}}_{ij}-\dot{\tilde{\bv}}_{ji}=-a_{ij}-a_{ji}=-2a_{ij}.\] 
If we set
\[ \label{E-6}
a_{ij}:=\frac{1}{2}(\kappa_1\dot{e}_{ij}+ \kappa_2{e}_{ij}),
\]
then one has
\[\ddot{e}_{ij}+ \kappa_1\dot{e}_{ij}+ \kappa_2{e}_{ij}=0.\] 
This implies that $e_{ij}$ converges to zero asymptotically (exponentially). Finally, we substitute \eqref{E-3} and \eqref{E-5} into \eqref{E-4} to find
\begin{equation} \label{Ap-6} 
a_{ij}=-\frac{\kappa_1}{2d_{ij}}g_{{\scriptsize \bx_i}}(P_{ij}\bv_j-\bv_i,\log_{\bx_i}\bx_j) +\frac{\kappa_2}{2}(R^\infty_{ij}-d_{ij}).
\end{equation}
Finally, we combine \eqref{Ap-5} and \eqref{Ap-6} to get the desired bonding force.

\bibliographystyle{amsplain}

\end{document}